\documentclass[12pt]{amsart}
\usepackage{amssymb,amsmath,amsthm,bm}
\usepackage[colorinlistoftodos,prependcaption,textsize=tiny]{todonotes}
 \definecolor{darkblue}{RGB}{0,0,160}
\usepackage{fouriernc} 
\usepackage[colorlinks=true,allcolors=darkblue]{hyperref}
\DeclareSymbolFont{usualmathcal}{OMS}{cmsy}{m}{n}
\DeclareSymbolFontAlphabet{\mathcal}{usualmathcal}

\usepackage[T1]{fontenc}

\usepackage{color}

\usepackage{parskip}
\usepackage{setspace}

\usepackage{amsmath,amsthm,amscd,amssymb, mathrsfs}

\usepackage{latexsym}

\numberwithin{equation}{section}

\theoremstyle{plain}
\newtheorem{theorem}{Theorem}[section]
\newtheorem{lemma}[theorem]{Lemma}
\newtheorem{corollary}[theorem]{Corollary}
\newtheorem{proposition}[theorem]{Proposition}
\newtheorem{conjecture}[theorem]{Conjecture}

\theoremstyle{definition}

\newtheorem{problem}[theorem]{Problem}

\theoremstyle{remark}
\newtheorem{remark}[theorem]{Remark}

\newtheorem{case[theorem]}{Case}

\title[\parbox{14cm}{\centering{Cone restriction conjecture and Applications\hspace{1in}}} \quad]{On the finite field cone restriction conjecture in four dimensions and applications in incidence geometry}
\author{Doowon Koh, Sujin Lee, and Thang Pham}

\address{Department of Mathematics\\
Chungbuk National University \\
Cheongju, Chungbuk 28644 Korea}
\email{koh131@chungbuk.ac.kr}

\address{Department of Mathematics\\
Chungbuk National University \\
Cheongju, Chungbuk 28644 Korea}
\email{sujin4432@chungbuk.ac.kr}

\address{Department of Mathematics, ETH Zurich, Switzerland}
\email{phamanhthang.vnu@gmail.com}


\begin{document}
\begin{abstract} 
The first purpose of this paper is to solve completely the finite field cone restriction conjecture in four dimensions  with $-1$ non-square.  The second is to introduce a new approach to study incidence problems via restriction theory. More precisely, using the cone restriction estimates, we will prove sharp point-sphere incidence bounds associated with  complex-valued functions for sphere sets of small size.  Our incidence bounds with a specific function improve significantly a result given by Cilleruelo, Iosevich, Lund, Roche-Newton, and Rudnev.
 \end{abstract}

\maketitle
\section{Introduction} 
Let $\mathbb F^n$ be an $n$-dimensional vector space over a finite field $\mathbb F.$ We assume that the characteristic of $\mathbb F$ is greater than two. We endow the vector space $\mathbb F^n$ with  counting measure $dx$, and its dual space $\mathbb F_*^n$ with  normalized counting measure $d\xi.$ Throughout this paper, we will denote by $e: \mathbb F \to \mathbb S^1$ the canonical additive character of $\mathbb F.$ For example, if $\mathbb F$ is a prime field of order $p,$ then we have
$e(t)=e^{2\pi i t/p}$. If $|\mathbb F|=p^\ell$ for some odd prime $p$, then we take
$e(t)=e^{2\pi i Tr(t)/p}$ for $t\in \mathbb F,$
where $|\mathbb F|$ denotes the cardinality of  $\mathbb F$ and $Tr$ is the trace function from $ \mathbb F$ to its subfield.

Given a complex-valued function $g$ on $\mathbb F^n$,  its Fourier transform $\widehat{g}$ is defined on the dual space $\mathbb F_*^n$ of the space $\mathbb F^n$ as follows:
\begin{equation}\label{defhat} \widehat{g}(\xi):= \sum_{x\in \mathbb F^n}  e(-\xi\cdot x) g(x).\end{equation} 
On the other hand,  given a complex-valued function $f: \mathbb F_*^n \to \mathbb C,$ the inverse Fourier transform of $f$, denoted by $f^\vee,$ is defined by
\begin{equation}\label{defvee} f^\vee(x):= \frac{1}{|\mathbb F^n|} \sum_{\xi\in \mathbb F_*^n} e(\xi\cdot x) f(\xi).\end{equation}
It is easy to check that $(\widehat{g})^\vee = g$ and $\widehat{(f^\vee)}=f$ which give  us the following Fourier inversion formulas:
$$ f(\xi)=\sum_{x\in \mathbb F^n} e(-\xi\cdot x) f^\vee(x)\quad \mbox{and}\quad g(x)=\frac{1}{|\mathbb F|^n} \sum_{\xi\in \mathbb F_*^n} e(\xi\cdot x) \widehat{g}(\xi).$$

By a direct computation, one can prove  the following equations which are called Plancherel's theorem.
$$\|\widehat{g}\|_{L^2(\mathbb F_*^n, d\xi)}=\|g\|_{L^2(\mathbb F^n, dx)}\quad\mbox{and}\quad \|f^\vee\|_{L^2(\mathbb F^n, dx)}=\|f\|_{L^2(\mathbb F_*^n, d\xi)}.$$
We now restrict our attention to the restriction problem for the cone $C_n$ in $\mathbb F_*^n.$ Here, and throughout this paper, the cone $C_n$ in $\mathbb F_*^n, ~n\ge 3,$ is defined by  
$$ C_n:=\{\xi\in \mathbb F_*^n:\xi_{n-1}\xi_n= \xi_1^2+\cdots+\xi_{n-2}^2\}.$$ 

We endow the cone $C_n$ with  normalized surface measure $d\sigma$. The normalized surface measure $d\sigma$ on $C_n$ assigns a mass of $|C_n|^{-1}$ to each point of the cone $C_n$ so that 
$$ \int_{\xi\in C_n} f(\xi) d\sigma(\xi)= \frac{1}{|C_n|} \sum_{\xi\in C_n} f(\xi),$$
and the inverse Fourier transform of a measure $fd\sigma$ is defined by
$$ (fd\sigma)^\vee(x)=\frac{1}{|C_n|} \sum_{\xi\in C_n} f(\xi) e(x\cdot \xi).$$

For $1\le p,r\le \infty$,  we denote by $R_{C_n}^*(p\to r)$ the best constant such that 
$$ \|(fd\sigma)^\vee\|_{L^r(\mathbb F^n, dx)} \le R_{C_n}^*(p\to r) \|f\|_{L^p(C_n, d\sigma)},$$
where the constant $R_{C_n}^*(p\to r)$ is independent of the functions $f$ on $C_n$, but it may depend on the size of the underlying finite field $\mathbb F.$
By duality, the above inequality is the same as the following restriction estimate:
\begin{equation}\label{dualEx} \|\widehat{g}\|_{L^{p'}(C_n, d\sigma)} \le R_{C_n}^*(p\to r) \|g\|_{L^{r'}(\mathbb F^n, dx)},\end{equation}
where $p', r'$ denote the H\"{o}lder conjugates of $p, r$, respectively (i.e.  $p'=p/(p-1), ~r'=r/(r-1)$). We will write $R_{C_n}^*(p\to r) \lesssim 1$ if $R_{C_n}^*(p\to r)$ is independent of the field size.\footnote{We use $X\gtrsim Y$ if $ CX\ge Y$ for some constant $C$ independent of $|\mathbb F|$. The notation $X\lesssim Y$ means that $Y\gtrsim X.$ In addition, we use $X\sim Y$ if $X\gtrsim Y$ and $X\lesssim Y.$ }
\begin{problem} [Restriction Problem]  Determine all pairs $(p,r)$ such that  $1\le p, r\le \infty$ and $R_{C_n}^*(p\to r)\lesssim 1.$
\end{problem}

Mockenhaupt and Tao \cite{MT04} were the first to study and solve  this problem in three dimensions. More precisely,  they proved that $R_{C_3}^*(p\to r)\lesssim 1$ if and only if $(1/p, 1/r)$ lies on the convex hull of the points $(0,0), ~(1,0), ~(1/2, 1/4), ~(0, 1/4).$ 

However, in higher dimensions, the problem is still wide open and a concrete statement of necessary conditions is not known. In the following lemma, we provide such conditions for the boundedness of $R^*_{C_n}(p\to r)$ (see Figure \ref{figure}). Its proof will be given in Section \ref{Sec2}.
\begin{lemma}\label{neceCone} Suppose that $R_{C_n}^*(p\to r)\lesssim 1.$ Then the following statements hold:
\begin{enumerate}
\item If $n\ge 3$ is odd, then $(1/p, 1/r)$ lies in the convex hull of the points $(0,0), (1,0),$\\ 
$\left( \frac{n^2-3n+4}{2n^2-4n+2},~ \frac{n-2}{2n-2}\right),$ and $\left(0,~ \frac{n-2}{2n-2}\right).$
 
\item If $n\equiv 2 \mod 4,$ or $n$ is even  and $-1\in  \mathbb F$ is a square number, then
$(1/p, 1/r)$ lies in the convex hull of the points $(0,0), (1,0), \left( \frac{n-2}{2n-2},~ \frac{n-2}{2n-2}\right)$ and $\left(0,~ \frac{n-2}{2n-2}\right).$

\item If $n\equiv 0 \mod 4$ and $-1\in  \mathbb F$ is not a square number, then $(1/p, 1/r)$ lies in the convex hull of the points $(0,0), (1,0), 
\left( \frac{n^2-2n+4}{2n^2-2n},~ \frac{n-2}{2n-2}\right),$ and
$\left(0,~\frac{n-2}{2n-2}\right).$
\end{enumerate}
\end{lemma}  
It is conjectured that the above conditions are also sufficient for $R_{C_n}^*(p\to r)\lesssim 1$ to hold. 
Notice that if we obtain a $R_{C_n}^*(p\to r)$ result corresponding to  the point $(1/p, 1/r)$  in $[0,1]\times [0,1]$, then we can use  H\H{ol}der's inequality and interpolation with the trivial estimate $R_{C_n}^*(1\to \infty)\lesssim 1$ to get all estimates corresponding to all points in the convex hull of the points $(0,0), (1,0), \left( {1}/{p},~ {1}/{r}\right)$ and $\left(0,~ {1}/{r}\right).$ Hence, to settle the finite field restriction problem for cones in $\mathbb F_*^n$, it will be enough to solve the cases of critical endpoints. This leads to the following conjecture.  
\begin{conjecture}\label{Coneconj} The following statements are true.
\begin{enumerate}
\item If $n\ge 3$ is odd, then 
$$ R_{C_n}^*\left( \frac{2n^2-4n+2}{n^2-3n+4} \to \frac{2n-2}{n-2}\right)\lesssim 1.$$
\item With the assumption of (2)  in Lemma \ref{neceCone}, we have 
$$R_{C_n}^*\left( \frac{2n-2}{n-2}\to \frac{2n-2}{n-2}\right)\lesssim 1.$$
\item With the assumption of (3)  in Lemma \ref{neceCone}, we have
$$R_{C_n}^*\left( \frac{2n^2-2n}{n^2-2n+4} \to \frac{2n-2}{n-2}\right)\lesssim 1.$$
\end{enumerate}
\end{conjecture}
This conjecture can be described clearly in Figure \ref{figure}.  



It follows from Mockenhaupt and Tao's result that $R_{C_3}^*(2\to 4)\lesssim 1$ for the cone in three dimensions, which matches the first case of  Conjecture \ref{Coneconj} with $n=3$. 

In the first result of this paper, we confirm the cone restriction conjecture for four dimensions in the case when $-1\in \mathbb F$ is not a square. More precisely, we have the following theorem. 
\begin{theorem}\label{bosung}
Let $\mathbb{F}$ be a finite field. Suppose that $|\mathbb{F}|\equiv 3\mod 4$, then we have 
\[R^*_{C_4}\left(2\to 3\right)\lesssim 1.\]
\end{theorem}
It is often difficult to determine a method to prove restriction estimates, and even much harder for the case of critical points in the finite field setting. For instance, let $P_n$ be the paraboloid in $\mathbb{F}_*^n$ defined by $\xi_n=\xi_1^2+\cdots +\xi_{n-1}^2$, there are only two known  methods to study the boundedness of $R_{P_n}^*(p\to r)$. The first is the finite field Stein-Tomas argument introduced by Mockenhaupt and Tao in the same paper, which mainly relies on the Fourier decay of the associated variety. This method only gives us estimates on $R_{P_n}^*(2\to r)$ which in general do not cover critical endpoints. In particular, Mockenhaupt and Tao in \cite{MT04} proved that $R_{P_n}^*\left(2\to \frac{2n+2}{n-1}\right)\lesssim 1$, which holds for all dimensions, but is only sharp when either $n=4\ell-1$ and $-1$ is a square or $n=4\ell+1$ for $\ell \in \mathbb N.$  In three dimensions, assuming that  $-1$ is not a square, they also proved that $R^*_{P_3}\left(2\to \frac{18}{5}+\varepsilon \right)\lesssim 1$ for all $\varepsilon>0.$ While $R^*_{C_3}(2\to 4)\lesssim 1$ is optimal for the cone $C_3$, it has been conjectured for the case of the paraboloid $P_3$ with $-1$ non-square that $R^*_{P_3}(2\to 3)\lesssim 1$ (see \cite{MT04}). The second machinery, introduced by Mockenhaupt and Tao in \cite{MT04} and  developed by Lewko in \cite{Le13}, allows us to  reduce the problem of bounding $R_{P_n}^*(2\to r)$ to an energy structure which has a connection to  point-line incidences. Over last five years, this method has been used intensively to obtain new $L^2$ restriction estimates in three dimensions along with new point-line incidence bounds. The best current bound is $R_{P_3}^*(2\to 3.547...)\lesssim 1$ due to Lewko in \cite{Le20}, which strengthens the argument in the paper \cite{RS19} by Rudnev and Shkredov, in which in turn the bound $R^*_{P_3}\left( 2\to \frac{32}{9}\right)\lesssim 1$ was attained via the Stevens-De Zeeuw point-line incidence bound \cite{SZ17}. It has been mentioned in \cite{RS19} that with the best energy bound, the current method only gives us $R^*_{P_3}\left( 2\to \frac{10}{3}\right)\lesssim 1$ that is far from the conjecture. 

For the finite field cone problem in dimensions $n\ge 4,$ the Stein-Tomas argument is the only known technique in the literature. One can check that this technique with  the sharp Fourier decay of the surface measure associated to the cone $C_n$ (Corollary \ref{cor4.4}) gives us  $R^*_{C_n}\left(2\to \frac{2n}{n-2}\right) \lesssim 1$ in all even dimensions. Thus, when $n=4$, we will get $R^*_{C_4}\left(2\to 4\right)\lesssim 1$, which is much inferior compared to Theorem \ref{bosung}. 

The main idea in our method is that we are able to decompose the square of the $L^2$ norm $||\widehat{g}||_{L^2(C_n, d\sigma)}$ for characteristic functions $g$ into several terms in a way that we can determine the sign of each term, and the terms with large absolute values usually have negative sign. Furthermore, the sign of each term in the decomposition will depend on the sign of the corresponding Fourier transform, which will be determined with the help of the explicit form of the Gauss sum in our dimensions. It turns out that the $L^2$ norm $||\widehat{g}||_{L^2(C_n, d\sigma)}$, when $n\equiv 0\mod 4$ and $|\mathbb{F}|\equiv 3\mod 4$, can be very small compared to other dimensions. This is the most interesting perspective of our proofs. 

We also note that in four dimensions, for the paraboloid problem, Rudnev and Shkredov \cite{RS19} proved  the same result, namely,  $R^*_{P_4}(2\to 3)\lesssim 1$, which gives the sharp $``r$" index for  $R^*_{P_4}(2\to r)\lesssim 1$ to hold in this dimension. However, their result does not solve completely the conjecture in four dimensions since it does not imply the conjectured result $R^*_{P_4}(16/7\to 8/3) \lesssim 1$ for $P_4$ (see \cite[Conjecture 2.4]{Ko16}). Moreover,  their result only holds in the setting of prime fields. This assumption comes from their main tool which is  Rudnev's point-plane incidence bound \cite{R}. 

In the Euclidean setting, the cone restriction problem has a reputed history, which we are not going to present. However, it is necessary to mention that the problem in four dimensions has been settled by Wolff in \cite{cone1} since the 2000s. The five dimensional problem was recently solved by Ou and Wang in \cite{cone2} by using the polynomial partitioning method. We refer the interested reader to \cite{cone2} and references therein for more details. Compared to Wolff's proof in the Euclidean setting and Rudnev and Shkredov's argument for the paraboloid $P_4$, our method in the proof of Theorem \ref{bosung} is totally different and can be extended to higher dimensions. If $|\mathbb{F}|\equiv 3\mod 4$ and $n\equiv 0\mod 4$, we have the following extension of Theorem \ref{bosung}.

\begin{theorem}\label{mainthm} Let $\mathbb F$ be a finite field with $|\mathbb F|\equiv 3 \mod 4.$  If $n\equiv 0 \mod 4,$ then 
$$ R_{C_n}^*\left(2\to \frac{2n+4}{n}\right)\lesssim 1.$$
\end{theorem}

\begin{figure}[h!]
\includegraphics[width=0.8\textwidth]{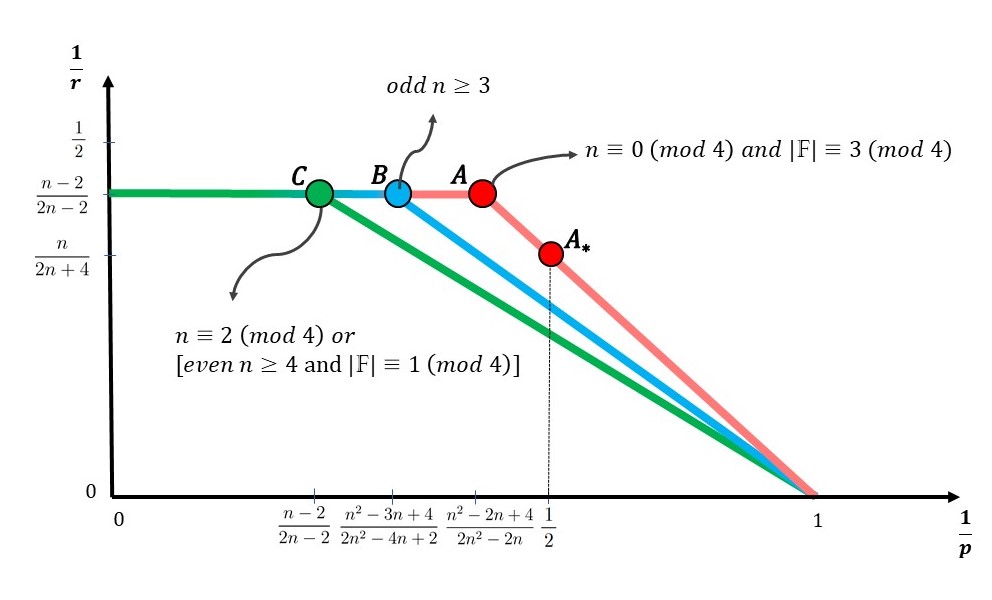}
\caption{ The points $A, B, C$ indicate the critical endpoints  for the boundedness of $R^*_{C_n} (p\to r),$ which are stated in Conjecture \ref{Coneconj}.  The point $A_*$ presents the result of 
of Theorem  \ref{mainthm}, and agrees with the critical endpoint $A$ for the case when $d=4$ and $|\mathbb{F}|\equiv 3\mod 4.$ }
\label{figure}
\end{figure}

One can check from Conjecture \ref{Coneconj} and Figure \ref{figure} that Theorem \ref{mainthm} gives us the sharp $``r"$ index for the $R_{C_n}(2\to r)$ bound when $n\equiv 0\mod 4$ and $|\mathbb{F}|\equiv 3\mod 4$, but it does not cover the critical endpoint for the conjecture when $n\ge 8$ (see, for example, Figure \ref{figure}). In other words, only in four dimensions, the conjecture is settled for the case when $-1\in \mathbb F$ is not a square.

For other dimensions, the $L^2$ sharp restriction estimates have been obtained by the Stein-Tomas argument in \cite[Theorem 2.1]{kohsen} by the first listed author and Shen. Actually, the $L^2$ restriction estimates in \cite{kohsen} were proved for a very general variety, which covers the case of the cone $C_n$.


\subsection{Applications in incidence geometry}

As we have seen,  incidence structures play the most important role in studying restriction problem for the paraboloid over finite fields, see also \cite{iooooo, KPV18, Le13, Le19, Le20, RS19, MT04}. There is also a numerous applications of incidence geometry in different topics. For instance, by using the Rudnev's point-plane incidence bound, Murphy, Petridis, Rudnev, Stevens, and the third listed author \cite{mu} obtained the exponent $5/4$ on the Erd\H{o}s-Falconer distance problem in two dimensions over prime fields, namely, they proved that for any set $P\subset \mathbb{F}^2$ with $|P|\gtrsim |\mathbb{F}|^{5/4}$,  the number of pinned distinct distances determined by $P$ is at least $c|\mathbb{F}|$ for some positive constant $c$. Notice that in the continuous setting, the exponent $5/4$ has been confirmed earlier for the Falconer distance problem in $\mathbb{R}^2$ by Guth, Iosevich, Ou, and Wang \cite{guth} by using the decoupling theory. There are also other important applications in additive combinatorics and theoretical computer science. We refer the interested reader to \cite{HH1, HH2, LEE, ROche} for more details. 

It is natural to ask if one can obtain new incidence bounds from restriction estimates, i.e. the reverse implication. The second purpose of this paper is on this direction. More precisely, we will use $L^2$ cone restriction estimates to derive new and optimal point-sphere incidence bounds in the finite field setting. This offers us a new way to think of incidence problems.

Before stating our results, we need to establish some notation. Given a vector $x=(x_1,\ldots, x_d)$ in $\mathbb F^d$, we define
$$ ||x||=x_1^2+\cdots+x_d^2.$$ 
Unlike the Euclidean case,  we regard $||x||$ as the distance between $x$ and the origin in $\mathbb F^d.$ 
Given $a\in \mathbb F^d, r\in \mathbb F,$ we define 
$$ S_d(a, r) :=\{x\in \mathbb F^d: ||x-a||=r\},$$ 
which will be named  the sphere of radius $r,$ centered at $a\in \mathbb F^d.$\\

 Let $P$ be a set of points in $\mathbb F^d$ and $S$ be a set of spheres in $\mathbb F^d$ with arbitrary radii. 

Given a complex function $w: S\to \mathbb C,$  the number of point-sphere incidences associated with the function $w$ is defined by
$$ I_w(P,S) := \sum_{p\in P, s\in S} 1_{p\in s} w(s), $$
where $1_{p\in s}=1$ if $p\in s$, and $0$ otherwise. When $w(s)=1$ for all $s\in S$, we write $I(P, S)$ instead of $I_w(P, S)$.

If one uses Theorem \ref{mainthm} as a black box, then we have 
\[\left\vert I_{w}(P, S)-|\mathbb{F}|^{-1}|P|\sum_{s\in S}w(s)\right\vert\lesssim |\mathbb{F}|^{\frac{d^2+3d-2}{2d+8}}|P|^{\frac{1}{2}}\left(\sum_{s\in S}|w(s)|^{\frac{2d+8}{d+6}}\right)^{\frac{d+6}{2d+8}},\]
for a point set $P\subset \mathbb{F}^d$ with 
$d\equiv 2\mod 4$ and $|\mathbb{F}|\equiv 3\mod 4$, and for  a collection $S$ of spheres in $\mathbb F^d$  (see Remark \ref{batky} for an explanation). 

However, by using the fact that the the support of the function $w$ is a subset of $S$, the $L^2$ Fourier restriction estimate with characteristic test functions gives us sharp exponents as follows.
\begin{theorem} \label{incidence-theorem}
Let $P$ be a set of points in $\mathbb F^d$ and $S$ be a set of spheres in $\mathbb F^d.$ Suppose that $w$ is a complex-valued function on $S.$ 
\begin{enumerate}
\item If $d\equiv 2 \mod{4}$, $|\mathbb F|\equiv 3 \mod{4}$, and $|S|\le |\mathbb{F}|^{\frac{d}{2}}$, then we have
$$\left|I_w(P, S)-|\mathbb F|^{-1} |P|\sum\limits_{s\in S} w(s) \right|\lesssim |\mathbb{F}|^{\frac{d-1}{2}}|P|^{\frac{1}{2}}\left(\sum_{s\in S } |w(s)|^2\right)^{\frac{1}{2}}.$$
\item If $d\equiv 0 \mod{4},$ or  $d$ is even and $|\mathbb F|\equiv 1 \mod{4}$, then the same conclusion holds under the condition $|S|\le |\mathbb{F}|^{\frac{d-2}{2}}$. 
\item If $d\ge 3$ is an odd integer, then the same conclusion holds under the condition $|S|\le |\mathbb{F}|^{\frac{d-1}{2}}$. 
\end{enumerate}
\end{theorem}
To see the sharpness of Theorem \ref{incidence-theorem}, we assume that $w(s)=1$ for all $s\in S$. So $I_w(P, S)$ will count  the usual point-sphere incidences. When the size of $S$ is small, say less than $|\mathbb{F}|^{\frac{d}{2}}, ~ |\mathbb{F}|^{\frac{d-2}{2}}, ~ |\mathbb{F}|^{\frac{d-1}{2}}$ in corresponding cases, then we have $\left\vert I(P, S)-|\mathbb{F}|^{-1}|P||S|\right\vert\lesssim |\mathbb{F}|^{\frac{d-1}{2}}\sqrt{|P||S|}$. This bound is sharp in the sense that there exist sets $P$ and $S$ in $\mathbb{F}^d$ with $|P||S|\sim |\mathbb{F}|^{d+1}$ and $|S|$  arbitrarily small such that $I(P, S)=0$. We will provide concrete constructions in Subsection \ref{chat}.

When $w(s)=1$ for all $s\in S$, Theorem \ref{incidence-theorem} improves significantly a result of 
Cilleruelo, Iosevich, Lund, Roche-Newton, and Rudnev \cite{CILRR17}, and independently by Pham, Phuong and Vinh \cite{PPV}, which says that $\left\vert I(P, S)-|\mathbb{F}|^{-1}|P||S|\right\vert\lesssim |\mathbb{F}|^{\frac{d}{2}}\sqrt{|P||S|}$. It is worth noting that it is not possible to prove Theorem \ref{incidence-theorem} by methods in \cite{CILRR17, PPV}.

In Theorem \ref{incidence-theorem} (1), if the incidence result under the condition $|S|\le |\mathbb F|^{d/2}$ can be  extended to the range $|S|\le |\mathbb F|^{(d+2)/2}$, then we will see in Subsection \ref{RemarkE} that one can prove the Erd\H{o}s-Falconer conjecture in the case when $d=4\ell+2$ for $\ell\in \mathbb N,$  and $-1$ is not a square number in $\mathbb F$. Hence,  in light of our results, this approach might be a feasible way to think of the Erd\H{o}s-Falconer distance conjecture over finite fields. 

We refer the interested reader to \cite{kollar, R, SZ17} for recent incidence bounds between points and lines, points and planes for small sets in the finite field setting. 



We will see in the next two sections that our proof of Theorem \ref{mainthm} relies heavily on discrete Fourier analysis and exponential sums. Thus, one might wonder whether or not it is possible to prove $L^2$ restriction estimates for cones by using incidence bounds like in the paraboloid setting. However, that question is outside the realm of methods of this paper.
\section{Fourier decay estimates on cones}

In this section, we recall properties of Gauss sums in \cite{LN97} and  an explicit form of the Fourier transform of the surface measure on the cone $C_n$ in \cite{KY}. We start with reviewing Gauss sums. 
Let $\eta:\mathbb F^*\rightarrow  \mathbb S^1$ be the quadratic character of $\mathbb F^*$, i.e., a group homomorphism defined by $\eta(t)=1$ if $t$ is a square, and $-1$ otherwise. It is not hard to see that $\eta$ has the following orthogonality property: for any $a\in  \mathbb F^*,$
$$ \sum_{t\in \mathbb F^*} \eta(at) =0.$$
For each $a\in \mathbb F^*$,  the Gauss sum $\mathcal{G}_a$ is defined by
$$ \mathcal{G}_a=\sum_{t\in  \mathbb F^*} \eta(t) e(at).$$
It is well--known  that $|\mathcal{G}_a|=|\mathbb F|^{1/2}$ for all $a\in  \mathbb F^*.$ 
Furthermore, the explicit value of the Gauss sum $\mathcal{G}_1$  is as follows.
\begin{lemma}\cite[Theorem 5.15]{LN97}\label{ExplicitGauss}
Let $\mathbb F$ be a finite field of order $p^{\ell},i.e., |\mathbb F|=p^\ell$, where $p$ is an odd prime and $\ell \in {\mathbb N}.$
Then we have
$$\mathcal{G}_1=\left\{\begin{array}{ll}  {(-1)}^{\ell-1} |\mathbb F|^{\frac{1}{2}} \quad &\mbox{if} \quad p \equiv 1 \mod 4 \\
                  {(-1)}^{\ell-1} i^\ell |\mathbb F|^{\frac{1}{2}}  \quad &\mbox{if} \quad p\equiv 3 \mod 4.\end{array}\right.$$
\end{lemma}

Thanks to the above explicit value of the Gauss sum 
$\mathcal{G}_1,$  we are able to deduce the following result which will provide crucial clues in proving the restriction conjecture for cones in dimension four. 
\begin{lemma}\label{Gsign} Let $\mathbb F$ be a finite field with $|\mathbb F|\equiv 3 \mod 4.$  If $n\equiv 0 \mod 4,$ then
$$ \mathcal{G}^{n-2}_1= -|\mathbb F|^{\frac{n-2}{2}}.$$ 
\end{lemma}
\begin{proof} Let $|\mathbb F|=p^\ell$ for some prime $p$ and a positive integer $\ell.$ Since $|\mathbb F|\equiv 3 \mod{4},$ we have
$p\equiv 3 \mod 4$ and $\ell$ is odd. Therefore, Lemma \ref{ExplicitGauss} implies that 
$$ \mathcal{G}_1 ={(-1)}^{\ell-1} i^\ell |\mathbb F|^{\frac{1}{2}}.$$
Since $n\equiv 0 \mod{4},$  we obtain that $\mathcal{G}_1^{n-2}=-|\mathbb F|^{\frac{n-2}{2}},$ as required.
\end{proof}

We will utilize the following properties of the Gauss sum which can be  proved by using a change of variables and properties of the quadratic character $\eta.$
For $a,b\ne 0$, we have
\begin{equation*}\label{Gauss}\sum_{s\in  \mathbb F^*} \eta(as) e(bs) =\sum_{s\in  \mathbb F^*} \eta(as^{-1}) e(bs)= \eta(ab) \mathcal{G}_1. 
\end{equation*}


Since $\eta(t)=1$ for a square number $t\in \mathbb F^*,$ and $-1$ otherwise,  we have for each $a \ne 0$ 
\begin{equation}\label{square} \sum_{s \in \mathbb F} e(a s^2) = \eta(a) \mathcal{G}_1.\end{equation}
For a proof of this equality, we refer readers to  \cite[Theorem $5.30$]{LN97}. 
 As a consequence of the equality \eqref{square}, it is not hard to see that  for $a\ne 0$ and $b\in \mathbb F,$
\begin{equation*}
\sum_{s \in \mathbb F} e(a s^2+bs) =
\eta(a) \mathcal{G}_1 e\left(\frac{b^2}{-4a}\right).  
    \end{equation*}
In fact, we have the following formula which  will be needed in our computations.
\begin{lemma}\label{complete}
For $\beta \in  \mathbb F^k$ and $s\in \mathbb F^*$, we have
$$ \sum_{\alpha \in \mathbb F^k} e( s \alpha \cdot \alpha + \beta \cdot \alpha ) 
=  \eta^k(s) \mathcal{G}_1^ke\left( \frac{\|\beta\|}{-4s}\right).$$
\end{lemma}
\begin{proof}
We have
$$\sum_{\alpha \in  \mathbb F^k} e( s \alpha \cdot \alpha + \beta \cdot \alpha ) 
              =\prod_{j=1}^{k} \sum_{\alpha_j\in \mathbb F} e( s\alpha_j^2 + \beta_j\alpha_j).$$
Therefore, the lemma follows by completing the square in the $\alpha_j$-variables, applying a change of variables, $ \alpha_j+\frac{\beta_j}{2s} \to \alpha_j$,
and using the equality in (\ref{square}). 
\end{proof}
We now recall some notations which are useful in stating the explicit form of $(1_{C_n})^\vee$, the inverse Fourier transform of the indicator function on the cone $C_n.$ For the simplicity, we will write   $C_n^\vee$ for $(1_{C_n})^\vee.$ 
For each $x=(x_1,\ldots, x_n)\in \mathbb F^n,$ we define
$$ \Gamma(x):=x_1^2+x_2^2+\cdots+x_{n-2}^2-4x_{n-1}x_n.$$
We also define 
$$ C_n^*:=\{x\in \mathbb F^n: \Gamma(x)=0\}.$$
The set $C_n^*$ can be considered as a dual variety of the cone $C_n$ in $\mathbb F_*^n.$ Recall that  $\delta_{\bf 0}(x)=1$ if $x=\bf 0$, and $0$ otherwise. With the  notation above in hand,  the explicit expression for  $ C_n^\vee(x):= |\mathbb F|^{-n} \sum_{\xi\in C_n} e(x\cdot \xi)$ was given in \cite{KY}. For the reader's convenience, we provide a proof here. 
\begin{proposition}\label{lem4.1} The following statements hold:
\begin{enumerate}
\item
If the dimension, $n\ge 4,$ is even,   then
$$ C_n^\vee(x)= \left\{\begin{array}{ll} \frac {\delta_{{\bf 0}}(x)}{|\mathbb F|}+\frac{(|\mathbb F|-1)\mathcal{G}_1^{n-2}}{|\mathbb F|^n} ~~ &\mbox{if} ~~x\in C_n^* \\
\frac{-\mathcal{G}_1^{n-2}}{|\mathbb F|^n}  ~~ &\mbox{if}~~ x\notin C_n^*. \end{array} \right.$$
\item
 If the dimension, $n\ge 3,$ is odd,   then
$$ C_n^\vee(x)= \left\{\begin{array}{ll} \frac {\delta_{{\bf 0}}(x)}{|\mathbb F|}~~ &\mbox{if} ~~ x\in C_n^*\\
 \frac{\mathcal{G}_1^{n-1}}{|\mathbb F|^n}~ \eta(-\Gamma(x)) ~~ &\mbox{if}~~ x\notin C_n^*. \end{array} \right.$$
\end{enumerate}
\end{proposition}

\begin{proof}   By the orthogonality of $e(\cdot),$ we have
\begin{align*} C_n^\vee(x)&=|\mathbb F|^{-n} \sum_{\xi\in C_n} e(x\cdot \xi)\\
                        &=\frac{\delta_{{\bf 0}}(x)}{|\mathbb F|} + \frac{1}{|\mathbb F|^{n+1}}\sum_{s\ne 0}  \sum_{\xi\in \mathbb F_*^n}e(s(\xi_1^2+\dots+\xi_{n-2}^2-\xi_{n-1}\xi_n)) e(x\cdot \xi).
\end{align*}
 Using the formula \eqref{square}, the above value becomes
$$  \frac{\delta_{{\bf 0}}(x)}{|\mathbb F|}+ \frac{\mathcal{G}_1^{n-2}}{|\mathbb F|^{n+1}} \sum_{s\ne 0} \eta^{n-2}(s)e\left( \frac{x_1^2+\cdots + x_{n-2}^2}{-4s}\right)  I (x_{n-1}, x_n),$$
where we define
$$ I(x_{n-1}, x_n):= \sum_{\xi_{n-1}\in  \mathbb F} e(x_{n-1}\xi_{n-1}) \sum_{\xi_n\in  \mathbb F} e( (-s\xi_{n-1}+x_n) \xi_n ).$$
By using the orthogonality of $e(\cdot),$  we compute the sum over $\xi_n\in  \mathbb F.$   Then  we   obtain  the following equation:
$$ C_n^\vee(x)=\frac{\delta_{{\bf 0}}(x)}{|\mathbb F|}+ \frac{\mathcal{G}_1^{n-2}}{|\mathbb F|^{n}} \sum_{s\ne 0} \eta^{n-2}(s) e\left(\frac {x_1^2+\cdots+x_{n-2}^2-4 x_{n-1}x_n}{-4s}\right).$$

Since $\eta^{n-2} =1$ for  even $n\geq 4$ ,  the first statement of Proposition \ref{lem4.1} follows. To prove the second part of Proposition \ref{lem4.1},
we first note that since the dimension $n$ is odd,  $\eta^{n-2}(s)=\eta(s)=\eta(s^{-1}) $ for $s\neq 0.$ Therefore, when $x_1^2+\cdots+x_{n-2}^2-4x_{n-1}x_n=0$, the statement follows immediately from the orthogonality  of $\eta.$  On the other hand, when $x_1^2+\cdots+x_{n-2}^2-4x_{n-1}x_n\neq 0,$  the statement follows  from a change of variables, the definition of the Gauss sum, and properties of the quadratic character $\eta.$
\end{proof}

Applying  the above proposition  with  the fact that $|\mathcal{G}_1|=\sqrt{|\mathbb F|}, $  one can easily find  the absolute value of the inverse Fourier transform on the cone $C_n.$  
More precisely, we have the following result.
\begin{corollary}\label{cor4.3} Let $x$ be a non-zero vector in $\mathbb F^n.$ 
\begin{enumerate}
\item
If the dimension, $n\ge 4,$ is even,   then
$$ |C_n^\vee(x)|\sim \left\{\begin{array}{ll} |\mathbb F|^{-\frac{n}{2}} ~~ &\mbox{if} ~~\Gamma(x)=0 \\
  |\mathbb F|^{- \frac{(n+2)}{2}}  ~~ &\mbox{if}~~ \Gamma(x)\ne 0. \end{array} \right.$$
\item
 If the dimension, $n\ge 3,$ is odd,   then
$$ |C_n^\vee(x)|= \left\{\begin{array}{ll} 0~~ &\mbox{if} ~~ \Gamma(x)=0\\
|\mathbb F|^{-\frac{(n+1)}{2}}~~ &\mbox{if}~~ \Gamma(x)\ne 0. \end{array} \right.$$
\end{enumerate}
\end{corollary}
It follows from Corollary \ref{cor4.3} that in even dimensions $|C_n^\vee(x)|$ is bounded by $|\mathbb{F}|^{-{n}/{2}}$ for any non zero $x$. However, when $n\equiv 0\mod 4$ and $|\mathbb{F}|\equiv 3\mod 4$, by taking the advantage of the sign of the Gauss sums in Lemma \ref{Gsign}, the upper bound of $C_n^\vee(x)$ can be much smaller. This is a key observation in the proof of Theorem \ref{mainthm}. An explicit form can be provided as follows. 

\begin{lemma}\label{Luck} If $|\mathbb F|\equiv 3 \mod{4}$ and $n\equiv 0 \mod{4},$ then
$$ C_n^\vee(x)= \left\{\begin{array}{ll} \frac {\delta_{{\bf 0}}(x)}{|\mathbb F|}-\frac{(|\mathbb F|-1)}{|\mathbb F|^{\frac{n+2}{2}}} ~~ &\mbox{if} ~~x\in C_n^* \\
\frac{1}{|\mathbb F|^{\frac{n+2}{2}}}  ~~ &\mbox{if}~~ x\notin C_n^*. \end{array} \right.$$
\end{lemma}
Let $d\sigma$ be the normalized surface measure on the cone $C_n.$  We also need a version of Corollary \ref{cor4.3} for this measure. 

\begin{corollary}\label{cor4.4} Let $d\sigma$ be the normalized surface measure on the cone $C_n.$ 
If $x$ is a non-zero vector in $\mathbb F^n$, then we have following:
\begin{enumerate}
\item
If the dimension, $n\ge 4,$ is even,   then
$$ |d\sigma^\vee(x)|\sim \left\{\begin{array}{ll} |\mathbb F|^{-\frac{(n-2)}{2}} ~~ &\mbox{if} ~~\Gamma(x)=0 \\
|\mathbb F|^{- \frac{n}{2}}  ~~ &\mbox{if}~~ \Gamma(x)\ne 0. \end{array} \right.$$
\item
 If the dimension, $n\ge 3,$ is odd,   then
$$ |d\sigma^\vee(x)|= \left\{\begin{array}{ll} 0~~ &\mbox{if} ~~ \Gamma(x)=0\\
  |\mathbb F|^{-\frac{(n-1)}{2}}~~ &\mbox{if}~~ \Gamma(x)\ne 0. \end{array} \right.$$
\end{enumerate}
\end{corollary}
\begin{proof} We first need to estimate the size of $C_n$ which can be calculated directly. Indeed, by the definition of $C_n^\vee$ and Proposition \ref{lem4.1}, it follows that 
$$ C_n^\vee(0,\ldots,0)= \frac{|C_n|}{|\mathbb F|^n} =\left\{ \begin{array}{ll} |\mathbb F|^{-1} + \frac{(|\mathbb F|-1)\mathcal{G}_1^{n-2}}{|\mathbb F|^n} \quad&\mbox{for even}~~n\ge 4\\                                       
 |\mathbb F|^{-1} \quad&\mbox{for odd}~~n\ge 3.
\end{array}\right.$$
Since $|\mathcal{G}_1|=\sqrt{|\mathbb F|}$,  we conclude 
$$|C_n|\sim |\mathbb F|^{n-1} \quad\mbox{for}~~n\ge 3.$$
Since $d\sigma^\vee(x)=\frac{|\mathbb F|^n}{|C_n|} C_n^\vee(x)$, the statement follows  immediately  from Corollary \ref{cor4.3}.
\end{proof}

\section{Necessary conditions (Proof of Lemma \ref{neceCone})}\label{Sec2}
Mockenhaupt and Tao \cite{MT04} introduced  necessary conditions for the boundedness of extension operators related to a general algebraic variety $V_n$ in $\mathbb F^n_*. $  For example, if $|V_n|\sim |\mathbb F|^{n-1}$ and $V_n$ contains a subspace  $H$ with $|H|=|\mathbb F|^k$,  then  the inequality,
\begin{equation}\label{neko1} r\ge \frac{p(n-k)}{(p-1)(n-1-k)},\end{equation}
is a necessary condition for the boundedness of $R_{V_n}^*(p\to r).$  Note that the RHS of (\ref{neko1}) is an increasing function in the dimension $k$. 

To understand more about this condition, we need to find a subspace of maximal size lying in the cone $C_n.$
To do this, we invoke the following well-known result (for a simple proof, see \cite[Lemma 2.1]{Vi12}).

\begin{lemma}\label{Vi} For an integer $n\ge 3,$ let 
$ S_{n-2}:=\{(x_1,\ldots, x_{n-2})\in \mathbb F^{n-2}: x_1^2+\cdots+x_{n-2}^2=0\}$ be a sphere with zero radius in 
$\mathbb F^{n-2}.$ 
Denote by $\eta$ the quadratic character of $\mathbb F^*.$
If $H$ is a subspace of maximal dimension contained in $S_{n-2}$, then the followings hold:
\begin{enumerate}
\item If $n-2$ is odd, then $|H|=|\mathbb F|^{\frac{n-3}{2}}.$
\item If $n-2$ is even and $(\eta(-1))^{\frac{n-2}{2}}=1$, then $|H|=|\mathbb F|^{\frac{n-2}{2}}.$
\item If $n-2$ is even and $(\eta(-1))^{\frac{n-2}{2}}=-1,$ then $|H|=|\mathbb F|^{\frac{n-4}{2}}.$
\end{enumerate}
\end{lemma}
Let $H$ be a subspace of maximal size lying in $S_{n-2},$ the sphere with zero radius in $\mathbb F^{n-2}.$
Define $\Omega=H\times \mathbb F\times \{0\}.$ We observe that $|\Omega|=|H||\mathbb F|$ and $\Omega$ is a subspace lying in the cone $C_{n}.$   
Combining this observation and Lemma \ref{Vi} gives us the next corollary.
\begin{corollary} \label{nekocor} The following statements hold:
\begin{enumerate} \label{SubP}
\item If $n\ge 3$ is odd, then the cone $C_n$ contains a subspace $\Omega$ with $|\Omega|=|\mathbb F|^{\frac{n-1}{2}}.$
\item If $n\equiv 2 \mod 4,$  then the cone $C_n$ contains a subspace $\Omega$ with $|\Omega|=|\mathbb F|^{\frac{n}{2}}.$
\item If $n$ is even  and $-1\in  \mathbb F$ is a square number, then the cone $C_n$ contains a subspace $\Omega$ with $|\Omega|= |\mathbb F|^{\frac{n}{2}}.$
\item If $n\equiv 0 \mod 4$, and $-1\in  \mathbb F$ is not a square number, then the cone $C_n$ contains a subspace $\Omega$ with $|\Omega|= |\mathbb F|^{\frac{n-2}{2}}.$
\end{enumerate}
\end{corollary}
We now can apply the inequality (\ref{neko1}) with the corresponding subspaces in Corollary \ref{nekocor} to get the following necessary conditions.
\begin{lemma}\label{nekor} Suppose that $R^*_{C_n}(p\to r)\lesssim 1.$ Then the following statements hold:
\begin{enumerate}
\item If $n\ge 3$ is odd, then 
$ r\ge \frac{p(n+1)}{(p-1)(n-1)}.  $
\item With the assumption of (2) or (3) in Corollary \ref{nekocor}, we have 
$r\ge \frac{pn}{(p-1)(n-2)}.$
\item If $n\equiv 0 \mod 4$, and $-1\in  \mathbb F$ is not a square number, then 
$ r\ge \frac{p(n+2)}{(p-1)n}.$
\end{enumerate}
\end{lemma}

In addition to these conditions, we have one more necessary condition as follows.

\begin{lemma} \label{konegood} If $R^*_{C_n}(p\to r)\lesssim 1,$ then we have 
$ r\ge \frac{2n-2}{n-2}.$
\end{lemma}
\begin{proof} Since $R^*_{C_n}(p\to r)\lesssim 1,$ we have
\begin{equation}\label{RestrictionEST}
 \|\widehat{g}\|_{L^{p'}(C_n, d\sigma)}\lesssim \|g\|_{L^{r'}(\mathbb F^d, dx)},
\end{equation}
Let $D := \{s \in \mathbb F^*: s ~\mbox{is a square number}\}.$ 
It is clear that $|D|=(|\mathbb F|-1)/2\sim |\mathbb F|.$ 
Next, define a set 
$$ \Gamma :=\left\{x \in \mathbb F^{n-1}\times D: x_{n-1} =\frac{x_1^2+ x_2^2+\cdots + x_{n-2}^2}{4x_n} \right\}.$$ 
Notice that $ |\Gamma|=|\mathbb F|^{n-2} |D|\sim |\mathbb F|^{n-1}.$
We will test \eqref{RestrictionEST} with the characteristic function of the set $\Gamma.$ We have
\begin{equation}\label{EqKT} ||\Gamma||_{L^{r'}(\mathbb F^d, dx)}=|\Gamma|^{\frac{1}{r'}}\sim |\mathbb F|^{\frac{n-1}{r'}}. \end{equation}
Let us estimate the quantity $\|\widehat{\Gamma}\|_{L^{p'}(C_n, d\sigma)}.$
For each $\xi \in C_n$ with $\xi_{n-1}\ne 0,$ we have
\begin{align*} \widehat{\Gamma}(\xi)=&\sum_{x\in \Gamma} 
e(-x\cdot \xi)\\
=& \sum_{x_1,\ldots,x_{n-2}\in \mathbb{F}} \sum_{x_n\in D} e\left(P_{\xi}(x_1, \ldots, x_{n-2}, x_n)\right), \end{align*}
where $P_{\xi}(x_1, \ldots, x_{n-2}, x_n)=-x_1\xi_1-\cdots-x_{n-2}\xi_{n-2} -\left(\frac{x_1^2+\cdots+ x_{n-2}^2}{4x_n}\right)\cdot \xi_{n-1}-x_n \xi_n.$
Applying Lemma \ref{complete}, we see that for each $\xi\in C_n$ with $\xi_{n-1}\ne 0,$
$$\widehat{\Gamma}(\xi)=\mathcal{G}_1^{n-2}\eta^{n-2}(-\xi_{n-1})\sum_{x_n\in D} \eta^{n-2}(x_n^{-1}) e\left(x_n \left( \frac{\xi_1^2+\cdots+ \xi_{n-2}^2}{\xi_{n-1}}-\xi_n\right)\right).$$
Since $\eta^{n-2}(x_n^{-1})=1$ for $x_n\in D$, and $ \frac{\xi_1^2+\cdots+ \xi_{n-2}^2}{\xi_{n-1}}-\xi_n=0$ for $\xi\in C_n$ with $\xi_{n-1}\ne 0,$ we see that for $\xi\in C_n$ with $\xi_{n-1}\ne 0,$ we have
$$ |\widehat{\Gamma}(\xi)|=|\mathcal{G}_1^{n-2}| |D|\sim |\mathbb F|^{\frac{n}{2}}.$$
Hence, it follows that
$$ \|\widehat{\Gamma}\|_{L^{p'}(C_n, d\sigma)}\gtrsim \left(\frac{1}{|C_n|} \sum_{\xi\in C_n: \xi_{n-1}\ne 0} 
|\mathbb F|^{\frac{np'}{2}}\right)^{\frac{1}{p'}} \sim |\mathbb F|^{\frac{n}{2}}.$$
From this estimate and \eqref{EqKT} together with the restriction inequality \eqref{RestrictionEST},  we have
$$|\mathbb F|^{\frac{n}{2}} \lesssim |\mathbb F|^{\frac{n-1}{r'}}.$$
This implies that $\frac{n}{2}\le \frac{n-1}{r'},$ which is the same as the condition that  $r\ge \frac{2n-2}{n-2},$ 
as required.
\end{proof}

Combining Lemma \ref{nekor} and Lemma \ref{konegood} gives us Lemma \ref{neceCone}. Lemma \ref{neceCone} can also be described in Figure \ref{figure} with the frame of $x={1}/{p}$ and $y={1}/{r}$ axes. 


\section{Proof of Theorem \ref{mainthm} (Restriction result)}
 We will establish the $L^2$ restriction inequality \eqref{dualEx}.
As a standard approach given in \cite{IKL17} we first obtain  more accurate $L^2$ Fourier restriction estimates with characteristic test functions. In this step, the discrete Fourier machinery functions as a powerful mechanism. For general test functions,  we can use the pigeonhole principle.

\begin{lemma}\label{WL2}
Let $|\mathbb F|\equiv 3 \mod 4$ and  $n\equiv 0 \mod 4.$ Then, for any $G \subset \mathbb F^n,$ we have
$$\|\widehat{G}\|_{L^2(C_n, d\sigma)} 
\lesssim \left\{\begin{array}{ll}|\mathbb F|^{\frac{1}{2}} |G|^{\frac{1}{2}} \quad&\mbox{for}~~|\mathbb F|^{\frac{n+2}{2}} \le |G|\le |\mathbb F|^{n}\\
|\mathbb F|^{-\frac{n}{4}} |G| \quad&\mbox{for}~~ |\mathbb F|^{\frac{n}{2}} \le |G|\le |\mathbb F|^{\frac{n+2}{2}}\\
|G|^{\frac{1}{2}} \quad&\mbox{for}~~ 1 \le |G|\le |\mathbb F|^{\frac{n}{2}}.\end{array}\right.
$$
\end{lemma}

\begin{proof}
Since $\|\widehat{G}\|_{L^2(C_n, d\sigma)} \sim 
\left(|\mathbb F|^{-n+1} \sum\limits_{\xi\in C_n} 
|\widehat{G}(\xi)|^2\right)^{1/2},$ it will be enough to establish the following three inequalities: letting $M:=|\mathbb F|^{-n+1} \sum\limits_{\xi\in C_n} |\widehat{G}(\xi)|^2,$ 
\begin{align}\label{Bo1} M\lesssim& ~|\mathbb F||G| \quad\mbox{if}~~|\mathbb F|^{\frac{n+2}{2}} \le |G|\le |\mathbb F|^{n},\\
\label{Bo2}M\lesssim&~ |\mathbb F|^{-\frac{n}{2}}|G|^2\quad\mbox{if}~~ |\mathbb F|^{\frac{n}{2}} \le |G|\le |\mathbb F|^{\frac{n+2}{2}},\\
\label{Bo3}M\lesssim&~  |G| \quad\mbox{if}~~ 1 \le |G|\le |\mathbb F|^{\frac{n}{2}}.\end{align}
The inequality \eqref{Bo1} is simply given by Plancherel's theorem as follows:
$$M\le |\mathbb F|^{-n+1} \sum_{\xi\in \mathbb F_*^n} |\widehat{G}(\xi)|^2=|\mathbb F|^{-n+1}|\mathbb F|^n \sum_{x\in \mathbb F^n} |G(x)|^2 = |\mathbb F||G|.$$

Let us prove \eqref{Bo2} and \eqref{Bo3}. By the definitions of the Fourier transform \eqref{defhat} and the inverse Fourier transform \eqref{defvee}, we can easily check the following:
$$ \sum_{\xi\in C_n} |\widehat{G}(\xi)|^2 =|\mathbb F|^n \sum_{x,y\in G} C_n^\vee(x-y).$$
It therefore follows that
\begin{align*}M&=|\mathbb F| \sum_{x,y\in G} C_n^\vee(x-y)\\
&=|\mathbb F| \sum_{x,y\in G: x-y\in C_n^*} C_n^\vee(x-y)
+ |\mathbb F| \sum_{x,y\in G: x-y\notin C_n^*} C_n^\vee(x-y).\end{align*}
Now we apply Lemma \ref{Luck} so that we can replace $C_n^\vee(x-y)$ by an explicit value. Thus, we have
$$M=  \sum_{x,y\in G: x-y\in C_n^*} \delta_{{\bf 0}} (x-y) - \sum_{x,y\in G: x-y\in C_n^*} |\mathbb F|^{-\frac{n}{2}}(|\mathbb F|-1)
+ \sum_{x,y\in G: x-y\notin C_n^*} |\mathbb F|^{-\frac{n}{2}} $$
$$=: M_1+M_2+M_3.$$
It is easy to see that
$M_1= |G|.$  The second term $M_2$ is a negative real number which can be negligible. The third term $M_3$ is a non-negative real number less than 
$\sum_{x,y\in G} |\mathbb F|^{-\frac{n}{2}}= |\mathbb F|^{-\frac{n}{2}}|G|^2.$  In  conclusion, we obtain 
$$M\le |G|+ |\mathbb F|^{-\frac{n}{2}}|G|^2,$$
which clearly implies that inequalities \eqref{Bo2}, \eqref{Bo3} hold, as required.
\end{proof}
We have few words to say from consequences of Lemma \ref{WL2}.
For sets $G$ with the size $|G|\sim |\mathbb F|^{(n+2)/2}$, we have  
$ \|\widehat{G}\|_{L^2(C_n, d\sigma)} \lesssim ||G||_{L^{p_0}(\mathbb F^n, dx)}$ with $p_0=(2n+4)/(n+4).$ On the other hand, for all sets  $G$ without  $|G|\sim |\mathbb F|^{(n+2)/2}$ 
(i.e., $|G|\sim |\mathbb F|^{(n+2)/2\pm \varepsilon}$ for  $\varepsilon>0$), we have much better restriction estimate, namely,
$\|\widehat{G}\|_{L^2(C_n, d\sigma)} \lesssim ||G||_{L^{p}(\mathbb F^n, dx)}$ for some $p >p_0.$
Hence, it is natural to expect that  for general test functions $g$ on  $\mathbb F^n$, we have
$$ \|\widehat{g}\|_{L^2(C_n, d\sigma)} \lesssim ||g||_{L^{p_0}(\mathbb F^n, dx)},$$
which is the same as Theorem \ref{mainthm} by duality.
In the following subsection, it is shown that this actually holds.

We will utilize the following lemma.
\begin{lemma}\label{RefereeQ} Let $|\mathbb F|\equiv 3 \mod 4$ and  $n\equiv 0 \mod 4.$ 
Suppose that $g, \widetilde{g}$ are non-negative functions on $\mathbb F^n,$ and satisfy that $ \frac{g(x)}{2} \le \widetilde{g}(x) \le 2g(x)$ for all $x\in \mathbb F^n.$
Then we have
\[  \|\widehat{\widetilde{g}}\|_{L^2(C_n, d\sigma)}\sim \|\widehat{g}\|_{L^2(C_n, d\sigma)} .\]
\end{lemma}
\begin{proof} It suffices to prove that
\[ \sum_{\xi\in C_n} |\widehat{\widetilde{g}}(\xi)|^2  \sim  \sum_{\xi\in C_n} |\widehat{g}(\xi)|^2.\]
By the definition of the Fourier trasform, we have
$$ \sum_{\xi\in C_n} |\widehat{\widetilde{g}}(\xi)|^2 = \sum_{\xi\in C_n} |\sum_{x\in \mathbb F^n} \widetilde{g}(x) e(-\xi\cdot x)|^2
=\sum_{\xi\in C_n} \sum_{x,x'\in \mathbb F^n} \widetilde{g}(x)\widetilde{g}(x') e(\xi\cdot (x'-x)),$$
where we used the assumption that $\widetilde{g}$ is a real-valued function. By the definition of the inverse Fourier transform on the cone $C_n$, the above equation is equal to 
\begin{equation} \label{Revise1}\sum_{\xi\in C_n} |\widehat{\widetilde{g}}(\xi)|^2 =|\mathbb F|^n \sum_{x,x'\in \mathbb F^n} \widetilde{g}(x)\widetilde{g}(x') C_n^\vee(x'-x).\end{equation}
By the same argument, we also obtain that
\begin{equation}\label{Revise2}\sum_{\xi\in C_n} |\widehat{g}(\xi)|^2 =|\mathbb F|^n \sum_{x,x'\in \mathbb F^n} g(x)g(x') C_n^\vee(x'-x).\end{equation}
Since $g\sim \widetilde{g}$, it is clear that $ \widetilde{g}(x)\widetilde{g}(x') \sim g(x)g(x')$ for all $x,x'\in \mathbb F^n.$ Finally, notice from Lemma \ref{Luck} that 
the range of the inverse Fourier transform $C_n^\vee$ consists of three real numbers. Hence,  the value of the RHS in \eqref{Revise1} is comparable to  that in \eqref{Revise2}. This completes the proof.
\end{proof}

\subsection{Proof of Theorem \ref{mainthm}} 

 We want to prove 
\begin{equation*}\label{Red2k} \|(fd\sigma)^\vee\|_{L^{\frac{2n+4}{n}}( \mathbb{F}^n, dx)} \lesssim \|f\|_{L^{2}(C_n, d\sigma)}\end{equation*}
under the assumptions that $|\mathbb F|\equiv 3 \mod{4}$ and $n=4\ell$ for $\ell\in \mathbb N.$ 
In fact, by duality, we will prove the following restriction estimate:
\begin{equation*}
\|\widehat{g}\|_{L^2(C_n, d\sigma)} \lesssim \|g\|_{L^{\frac{2n+4}{n+4}}( \mathbb F^n, dx)} := \left( \sum_{x\in  \mathbb F^n} |g(x)|^{\frac{2n+4}{n+4}} \right)^{\frac{n+4}{2n+4}}.
\end{equation*}
To prove this inequality, without loss of generality, we can assume that 
\begin{equation}\label{kohgo}\sum\limits_{x\in  \mathbb F^n} |g(x)|^{\frac{2n+4}{n+4}}=1,\end{equation}
and only need to prove that $\|\widehat{g}\|_{L^2(C_n, d\sigma)} \lesssim 1.$ 
Furthermore, we may assume that $0\le g(x)\le 1 $ for all $x\in \mathbb F^n.$
Using Lemma \ref{RefereeQ}, we may assume that the function $g$ is written in the form
$$ g=\sum_{i=0}^\infty 2^{-i} 1_{G_i} =\sum_{i=0}^L 2^{-i} 1_{G_i}+ \sum_{i=L+1}^\infty 2^{-i} 1_{G_i}=: g_1+ g_2,$$
where $G_i:=\{x\in \mathbb F^n:  2^{-j-1}< g(x) \le 2^{-j}\}$ and $L$ is an integer with $L\ge n\log|\mathbb F|.$ 
It is not hard to see that 
$$\|\widehat{g_2}\|_{L^2(C_n, d\sigma)}\le \sum_{i=L+1}^\infty 2^{-i} ||\widehat{1_{G_i}}||_{L^2(C_n, d\sigma)} \le \sum_{i=L+1}^\infty 2^{-i}|G_i| \lesssim |\mathbb F|^n 2^{-(L+1)} \lesssim 1.$$ 
Therefore, we may assume that the function $g$ is written as
\begin{equation}\label{Ass2kN} g=\sum_{i=0}^L 2^{-i} 1_{G_i},\end{equation}
and want to show that $\|\widehat{g}\|_{L^2(C_n, d\sigma)}\lesssim 1.$
From \eqref{kohgo} and \eqref{Ass2kN}, we  have 
$ \sum_{i=0}^L 2^{-\frac{2n+4}{n+4}i} |G_i| =1.$ 
Hence, it follows that
\begin{equation}\label{GG}  |G_i|\le 2^{\frac{2n+4}{n+4} i} \quad \mbox{for all}~~i=0,1,2, \ldots, L. \end{equation}

In conclusion, our problem is reduced to showing that
$
\|\widehat{g}\|_{L^2(C_n, d\sigma)} \lesssim 1
$ under the assumption that $g$ satisfies both \eqref{Ass2kN} and \eqref{GG}. It follows by \eqref{Ass2kN} and Minkowski's inequality that
$$\|\widehat{g}\|_{L^2(C_n, d\sigma)} \le \sum_{i=0}^L 2^{-i} \|\widehat{1_{G_i}}\|_{L^2(C_n, d\sigma)}.$$
To estimate the sum in the RHS of the above inequality, we decompose the set $I:=\{0, 1,\ldots, L\}$ into three sets as follows:
$$ I_1:= \{i\in I: 1\le 2^{\frac{2n+4}{n+4} i}\le 
|\mathbb F|^{\frac{n}{2}}\}, \quad 
I_2:=\{i\in I: |\mathbb F|^{\frac{n}{2}}\le 2^{\frac{2n+4}{n+4} i}\le |\mathbb F|^{\frac{n+2}{2}}\}, \quad \mbox{and}$$
$$I_3:=\{i\in I: |\mathbb F|^{\frac{n+2}{2}}\le 2^{\frac{2n+4}{n+4} i}\le |\mathbb F|^{n}\}.$$
We then have
\begin{align*}\|\widehat{g}\|_{L^2(C_n, d\sigma)} 
&\le  \sum_{i\in I_1} 2^{-i} \|\widehat{1_{G_i}}\|_{L^2(C_n, d\sigma)} + 
\sum_{i\in I_2} 2^{-i} \|\widehat{1_{G_i}}\|_{L^2(C_n, d\sigma)} + 
\sum_{i\in I_3} 2^{-i} \|\widehat{1_{G_i}}\|_{L^2(C_n, d\sigma)} \\
&=: U_1+ U_2+U_3. 
\end{align*}

 Applying Lemma \ref{WL2} with \eqref{GG}, we obtain
 $$ U_1 \lesssim \sum_{i\in I_1} 2^{-i} |G_i|^{\frac{1}{2}}  \lesssim \sum_{i\in I_1} 2^{-i}2^{\frac{n+2}{n+4} i}  \lesssim 1,$$
  $$ U_2 \lesssim \sum_{i\in I_2} 2^{-i} |\mathbb F|^{-\frac{n}{4}}|G_i|
 \lesssim |\mathbb F|^{-\frac{n}{4}}\sum_{i\in I_2} 2^{-i}2^{\frac{2n+4}{n+4} i}  \lesssim |\mathbb F|^{-\frac{n}{4}} |\mathbb F|^{\frac{n}{4}} = 1,$$
and
  $$ U_3 \lesssim  \sum_{i\in I_3} 2^{-i} |\mathbb F|^{\frac{1}{2}} |G_i|^{\frac{1}{2}} 
 \lesssim  |\mathbb F|^{\frac{1}{2}} \sum_{i\in I_3} 2^{-i}2^{\frac{n+2}{n+4} i}  \lesssim |\mathbb F|^{\frac{1}{2}}|\mathbb F|^{-\frac{1}{2}}=1.$$
This completes the proof of Theorem \ref{mainthm}. 
$\hfill\square$

\section{Point-sphere incidence bounds}

Recall that   $S_d(a,r)$ denotes the sphere centered at $a\in \mathbb F^d$ of radius $r.$ We will identify the sphere $S_d(a,r)$ with $(a, r)$ in $\mathbb F^{d+1}.$
 Given a collection of spheres in $\mathbb F^d,$ denoted by $S$, we define 
 \begin{equation*}\label{defS'} S':=\{ t(-2a, 1, ||a||-r) \in \mathbb F^{d+2}: (a, r)\in S \subset \mathbb F^d\times \mathbb F,~ t\in \mathbb F^*\}.\end{equation*}
 Let $w$ be a complex-valued function supported on $S.$ We define 
a function  $w'$ on $S'$ by
$$ w'(t(-2a, 1, ||a||-r)):=w(a,r)$$
for $a\in \mathbb F^d, r\in \mathbb F, t\in \mathbb F^*.$ Notice that the lines in $S'$ are pairwise disjoint, and thus $w'$ is well-defined.
Recall that the number of point-sphere incidences associated with the function $w$ is defined by
$$ I_w(P,S):= \sum_{p\in P, s\in S} 1_{p\in s} w(s).$$

In the following proposition, we give a reduction from point-sphere incidences to  $L^2$ restriction estimate for cones in $\mathbb{F}^{d+2}$. 
\begin{proposition}\label{Pro15} Let $P$ be a set of points in $\mathbb F^d$ and $S$ be a set of spheres in $\mathbb F^d.$ For each complex-valued function $w$ on $S,$ we have
$$ \left|I_w(P, S)- \frac{|P|}{|\mathbb{F}|} \sum_{s\in S} w(s)\right| \lesssim  |\mathbb{F}|^{-\frac{3}{2}} |P|^{\frac{1}{2}} \left( \sum_{\mathbf{x}\in C_{d+2}} |\widehat{w'1_{S'}}(\mathbf{x})|^2\right)^{\frac{1}{2}}.$$
\end{proposition}
\begin{proof} It follows that
$$ I_w(P, S)= \sum_{\substack{x\in P, (a,r)\in S:\\ ||x-a||-r=0}} w(a,r)=\sum_{\substack{x\in P, (a,r)\in S:\\ (x, ||x||, 1)\cdot (-2a, 1, ||a||-r)=0}}w(a,r).$$

Let $P'=\{\lambda\cdot (x, ||x||, 1) \in \mathbb F^{d+2}: x\in P, \lambda \in \mathbb F^*\}.$ Since $|P'|=(|\mathbb{F}|-1)|P|,$ we have
$$ I_w(P,S)= (|\mathbb{F}|-1)^{-1} \sum_{\mathbf{x}\in P', (a,r)\in S: \mathbf{x}\cdot (-2a, 1, ||a||-r)=0} w(a,r).$$ By the orthogonality of $e$,  we have
$$ I_w(P,S)=\frac{|P'|}{|\mathbb{F}| (|\mathbb{F}|-1 )} \sum_{(a,r)\in S} w(a,r) + \frac{1}{|\mathbb{F}|(|\mathbb{F}|-1)} \sum_{\mathbf{x}\in P', (a,r)\in S}\sum_{t\in \mathbb F^*} e( \mathbf{x} \cdot t(-2a, 1, ||a||-r))w(a,r)$$ 
$$= \frac{|P|}{|\mathbb{F}|} \sum_{s\in S} w(s)  + \frac{1}{|\mathbb{F}|(|\mathbb{F}|-1)} \sum_{\mathbf{x}\in P', \mathbf{y}\in S'} e(\mathbf{x}\cdot \mathbf{y}) w'(\mathbf{y}).$$

This implies that
$$\left|I_w(P, S)- \frac{|P|}{|\mathbb{F}|} \sum_{s\in S} w(s)\right| \lesssim  |\mathbb{F}|^{-2}\sum_{\mathbf{x}\in P'} \left|\widehat{w'1_{S'}}(\mathbf{x})\right|.$$

Notice that $P'$ is a subset of the cone $C_{d+2}$ in $\mathbb F^{d+2}.$ By applying the Cauchy-Schwarz inequality, we get
$$ \left|I_w(P, S)- \frac{|P|}{|\mathbb{F}|} \sum_{s\in S} w(s)\right| \lesssim  
 |\mathbb{F}|^{-2} |P'|^{\frac{1}{2}} \left( \sum_{\mathbf{x}\in C_{d+2}} |\widehat{w'1_{S'}}(\mathbf{x})|^2\right)^{\frac{1}{2}}.$$
 Since $|P'|\le |\mathbb{F}| |P|$, we complete the proof. 
\end{proof}
\begin{remark}\label{batky}
If we apply Theorem \ref{mainthm} directly in  Proposition \ref{Pro15}, then we get the following
\begin{align*}
\left|I_w(P, S)- \frac{|P|}{|\mathbb{F}|} \sum_{s\in S} w(s)\right| &\lesssim |\mathbb{F}|^{\frac{d-2}{2}} |P|^{\frac{1}{2}}  \left(\sum_{\mathbf{y}\in S'}|w'(\mathbf{y})|^{\frac{2d+8}{d+6}} \right)^{\frac{d+6}{2d+8}}\\
&\lesssim |\mathbb{F}|^{\frac{d^2+3d-2}{2d+8}}|P|^{\frac{1}{2}}\left(\sum_{s\in S}|w(s)|^{\frac{2d+8}{d+6}}\right)^{\frac{d+6}{2d+8}}.
\end{align*}
However, based on the fact that the support of $w$ is a subset of $S$, we can bound it in a much better way by using the following observation
\begin{equation*}\sum_{\mathbf{x}\in C_{d+2}} |\widehat{w'1_{S'}}(\mathbf{x})|^2= |\mathbb{F}|^{d+2} \sum_{\mathbf{m}, \mathbf{m'}\in S'}  w'(\mathbf{m}) \overline{w'(\mathbf{m'})} C_{d+2}^{\vee}(\mathbf{m}-\mathbf{m'}),\end{equation*}
which says that the point-sphere incidence problem associated with the function $w$ on a family of spheres in $\mathbb F^d$ has a connection with the Fourier decay on the cone $C_{d+2}$ in $\mathbb F^{d+2}.$ \\
This leads us to the following result.
\end{remark}

\begin{proposition} \label{Goodsize}Let $P$ be a set of points in $\mathbb F^d$ and $S$ be a set of spheres in $\mathbb F^d.$
Then, for each complex-valued function $w$ on $S,$ the following statements hold:

\begin{enumerate}
\item If $d\equiv 2 \mod{4}$ and $|\mathbb{F}|\equiv 3 \mod{4}$, then
$$\sum_{\mathbf{x}\in C_{d+2}} |\widehat{w'1_{S'}}(\mathbf{x})|^2 
\lesssim \left\{ \begin{array}{ll} |\mathbb{F}|^{d+3} \sum\limits_{s\in S} |w(s)|^2 \quad &\mbox{for}\quad |\mathbb{F}|^{\frac{d+2}{2}}\le |S| \le |\mathbb{F}|^{d+1}\\
                              |\mathbb{F}|^{\frac{d+4}{2}} |S|\sum\limits_{s\in S}  |w(s)|^2  \quad &\mbox{for}\quad |\mathbb{F}|^{\frac{d}{2}}\le |S| \le |\mathbb{F}|^{\frac{d+2}{2}}\\
                             |\mathbb{F}|^{d+2} \sum\limits_{s\in S}  |w(s)|^2  \quad &\mbox{for}\quad 1\le |S| \le |\mathbb{F}|^{\frac{d}{2}}. \end{array}\right. 
                                                  $$
\item If $d\equiv 0 \mod{4},$ or $d$ is even and $|\mathbb{F}|\equiv 1 \mod{4}$, then
$$\sum_{\mathbf{x}\in C_{d+2}} |\widehat{w'1_{S'}}(\mathbf{x})|^2
\lesssim \left\{ \begin{array}{ll} |\mathbb{F}|^{d+3} \sum\limits_{s\in S}  |w(s)|^2\quad &\mbox{for}\quad |\mathbb{F}|^{\frac{d}{2}}\le |S| \le |\mathbb{F}|^{d+1}\\
                              |\mathbb{F}|^{\frac{d+6}{2}} |S|\sum\limits_{s\in S}  |w(s)|^2 \quad &\mbox{for}\quad |\mathbb{F}|^{\frac{d-2}{2}}\le |S| \le |\mathbb{F}|^{\frac{d}{2}}\\
                             |\mathbb{F}|^{d+2} \sum\limits_{s\in S}  |w(s)|^2\quad &\mbox{for}\quad 1\le |S| \le |\mathbb{F}|^{\frac{d-2}{2}}. \end{array}\right.                      $$

\item If $d\ge 3$ is an odd integer, then
$$\sum_{\mathbf{x}\in C_{d+2}} |\widehat{w'1_{S'}}(\mathbf{x})|^2
\lesssim \left\{ \begin{array}{ll} |\mathbb{F}|^{d+3} \sum\limits_{s\in S}  |w(s)|^2 \quad &\mbox{for}\quad |\mathbb{F}|^{\frac{d+1}{2}}\le |S| \le |\mathbb{F}|^{d+1}\\
                              |\mathbb{F}|^{\frac{d+5}{2}} |S|\sum\limits_{s\in S}  |w(s)|^2 \quad &\mbox{for}\quad |\mathbb{F}|^{\frac{d-1}{2}}\le |S| \le |\mathbb{F}|^{\frac{d+1}{2}}\\
                             |\mathbb{F}|^{d+2} \sum\limits_{s\in S}  |w(s)|^2 \quad &\mbox{for}\quad 1\le |S| \le |\mathbb{F}|^{\frac{d-1}{2}}. \end{array}\right.                      $$
\end{enumerate}

\end{proposition}

\begin{proof}  Without loss of generality, we may assume that the function $w$  is  real-valued  since  a general complex-valued function $w$ can be written by $w=w_1+ i w_2$ for some real-valued functions $w_1, w_2$ on $S,$  and it satisfies that 
$$\sum_{\mathbf{x}\in C_{d+2}} |\widehat{w'1_{S'}}(\mathbf{x})|^2 \lesssim \sum_{\mathbf{x}\in C_{d+2}} |\widehat{w_1'1_{S'}}(\mathbf{x})|^2 + \sum_{\mathbf{x}\in C_{d+2}} |\widehat{w_2'1_{S'}}(\mathbf{x})|^2$$ and 
$$ \max\left\{ \sum\limits_{s\in S}  |w_1(s)|^2,~ \sum\limits_{s\in S}  |w_2(s)|^2\right\} \le \sum\limits_{s\in S}  |w(s)|^2.$$

 Furthermore, without loss of generality, we may assume that
 the function $w$  is  non-negative.  To see this, notice that a  real-valued function $w$ can be written by
 $w=w^+-w^{-}$, where $w^+$ and $w^{-}$ are non-negative real-valued functions on $S$ defined as follows:
 $$ w^+(s)=\left\{\begin{array}{ll} w(s)\quad &\mbox{if}~~w(s)\ge 0\\
 0 \quad &\mbox{if}~~w(s)<0,\end{array}\right. \quad{and}\quad w^-(s)=\left\{\begin{array}{ll} 0 \quad &\mbox{if}~~w(s)\ge 0\\
 |w(s)| \quad &\mbox{if}~~w(s)<0.\end{array}\right.$$
In addition, observe  $$\sum_{\mathbf{x}\in C_{d+2}} |\widehat{w'1_{S'}}(\mathbf{x})|^2 \lesssim \sum_{\mathbf{x}\in C_{d+2}} |\widehat{(w^+)'1_{S'}}(\mathbf{x})|^2 + \sum_{\mathbf{x}\in C_{d+2}} |\widehat{(w^-)'1_{S'}}(\mathbf{x})|^2,$$ 
and 
$$ \max\left\{ \sum\limits_{s\in S}  |w^+(s)|^2,~ \sum\limits_{s\in S}  |w^-(s)|^2\right\} \le \sum\limits_{s\in S}  |w(s)|^2.$$

Hence, it suffices to prove the proposition by assuming that $w(s)\ge 0$ for all $s\in S$ and so $w'\ge 0.$ Let us prove the first part of the proposition. Assume that $d\equiv 2 \mod{4}$ and $|\mathbb{F}|\equiv 3 \mod{4}.$
By definitions of the Fourier transform  and the inverse Fourier transform, we have
$$\sum_{\mathbf{x}\in C_{d+2}} |\widehat{w'1_{S'}}(\mathbf{x})|^2= |\mathbb{F}|^{d+2} \sum_{\mathbf{m}, \mathbf{m'}\in S'}  w'(\mathbf{m}) \overline{w'(\mathbf{m'})} C_{d+2}^{\vee}(\mathbf{m}-\mathbf{m'})  $$
$$ =|\mathbb{F}|^{d+2} \sum_{\substack{\mathbf{m}, \mathbf{m'}\in S'\\:\mathbf{m}-\mathbf{m'}\in C_{d+2}^* }}w'(\mathbf{m}) \overline{w'(\mathbf{m'})} C_{d+2}^{\vee}(\mathbf{m}-\mathbf{m'}) 
+|\mathbb{F}|^{d+2} \sum_{\substack{\mathbf{m}, \mathbf{m'}\in S'\\
:\mathbf{m}-\mathbf{m'}\notin C_{d+2}^*} } w'(\mathbf{m}) \overline{w'(\mathbf{m'})} C_{d+2}^{\vee}(\mathbf{m}-\mathbf{m'}).$$
Since $|\mathbb{F}|\equiv 3 \mod{4}$ and $d\equiv 2 \mod{4}$,  it follows from Lemma \ref{ExplicitGauss}  that 
$$ \mathcal{G}_1^{d}= -|\mathbb{F}|^{\frac{d}{2}}.$$
Hence, by Lemma \ref{lem4.1} with $n=d+2$, we have 
$$ C_{d+2}^\vee(\mathbf{m}-\mathbf{m'})= \left\{\begin{array}{ll} \frac {\delta_{{\bf 0}}(\mathbf{m}-\mathbf{m'})}{|\mathbb{F}|}-\frac{(|\mathbb{F}|-1)|\mathbb{F}|^{\frac{d}{2}}}{|\mathbb{F}|^{d+2}} ~~ &\mbox{if} ~~\mathbf{m}-\mathbf{m'}\in C_{d+2}^* \\
\frac{|\mathbb{F}|^{\frac{d}{2}}}{|\mathbb{F}|^{d+2}}  ~~ &\mbox{if}~~ \mathbf{m}-\mathbf{m'}\notin C_{d+2}^*. \end{array} \right.$$
Inserting this estimate into the above equality,  we have
$$\sum_{\mathbf{x}\in C_{d+2}} |\widehat{w'1_{S'}}(\mathbf{x})|^2 =|\mathbb{F}|^{d+2} \sum_{\substack{\mathbf{m}, \mathbf{m'}\in S':\\
\mathbf{m}-\mathbf{m'}\in C_{d+2}^*}}w'(\mathbf{m}) \overline{w'(\mathbf{m'})} \left(\frac {\delta_{{\bf 0}}(\mathbf{m}-\mathbf{m'})}{|\mathbb{F}|}-\frac{(|\mathbb{F}|-1)|\mathbb{F}|^{\frac{d}{2}}}{|\mathbb{F}|^{d+2}}\right)$$
$$+|\mathbb{F}|^{d+2} \sum_{\substack{\mathbf{m}, \mathbf{m'}\in S':\\\mathbf{m}-\mathbf{m'}\notin C_{d+2}^*} }w'(\mathbf{m}) \overline{w'(\mathbf{m'})}\frac{|\mathbb{F}|^{\frac{d}{2}}}{|\mathbb{F}|^{d+2}}. $$
Simplifying the above sums,  this value is the same as
 $$|\mathbb{F}|^{d+1}\sum_{\mathbf{m}\in S'} |w'(\mathbf{m})|^2 -(|\mathbb{F}|-1)|\mathbb{F}|^{\frac{d}{2}}\sum_{\substack{\mathbf{m}, \mathbf{m'}\in S':\\
\mathbf{m}-\mathbf{m'}\in C_{d+2}^*}} w'(\mathbf{m}) \overline{w'(\mathbf{m'})}  +|\mathbb{F}|^{\frac{d}{2}} \sum_{\substack{\mathbf{m}, \mathbf{m'}\in S':\\\mathbf{m}-\mathbf{m'}\notin C_{d+2}^*} }w'(\mathbf{m}) \overline{w'(\mathbf{m'})}. $$
Note that the second term above is negative since $w'$ is a non-negative function. Hence,  we obtain that
  \begin{equation}\label{bor1}\sum_{\mathbf{x}\in C_{d+2}} |\widehat{w'1_{S'}}(\mathbf{x})|^2 \le |\mathbb{F}|^{d+1}\sum_{\mathbf{m}\in S'} |w'(\mathbf{m})|^2 + |\mathbb{F}|^{\frac{d}{2}} \left|\sum_{\mathbf{m}\in S'} w'(\mathbf{m})\right|^2.\end{equation}
We now apply the Cauchy-Schwarz inequality to the second term above and notice by the definition of $w'$ that 
$$ \sum_{\mathbf{m}\in S'} |w'(\mathbf{m})|^2 = (|\mathbb{F}|-1) \sum_{s\in S} w^2(s).$$ 
   Then we have
  $$ \sum_{\mathbf{x}\in C_{d+2}} |\widehat{w'1_{S'}}(\mathbf{x})|^2  \le \left(|\mathbb{F}|^{d+2} + |\mathbb{F}|^{\frac{d+4}{2}} |S| \right)\sum_{s\in S} w^2(s).$$
 On the other hand, we also obtain by Plancherel's theorem that
  \begin{equation}\label{Plancherel}\sum_{\mathbf{x}\in C_{d+2}} |\widehat{w'1_{S'}}(\mathbf{x})|^2  \le \sum_{\mathbf{x}\in \mathbb F^{d+2}} |\widehat{w'1_{S'}}(\mathbf{x})|^2  \le |\mathbb{F}|^{d+3} \sum\limits_{s\in S} w^2(s).\end{equation}
 By a direct computation, it is not hard to see that  the above two inequalities imply the first part of the proposition, as required.
 
To prove  the second part of the proposition,  we notice from Lemma \ref{ExplicitGauss}   that if $d\equiv 0 \mod{4}$ or $|\mathbb{F}|\equiv 1 \mod{4},$ then
$\mathcal{G}_1^{d}= |\mathbb{F}|^{\frac{d}{2}}.$ Following the same argument as in the proof of the first part of the proposition, we see that 
$\sum_{\mathbf{x}\in C_{d+2}} |\widehat{w'1_{S'}}(\mathbf{x})|^2$ is written as the following:
$$|\mathbb{F}|^{d+1}\sum_{\mathbf{m}\in S'} |w'(\mathbf{m})|^2 +(|\mathbb{F}|-1)|\mathbb{F}|^{\frac{d}{2}}\sum_{\substack{\mathbf{m}, \mathbf{m'}\in S':\\
\mathbf{m}-\mathbf{m'}\in C_{d+2}^*}} w'(\mathbf{m}) \overline{w'(\mathbf{m'})}  - \sum_{\substack{\mathbf{m}, \mathbf{m'}\in S':\\\mathbf{m}-\mathbf{m'}\notin C_{d+2}^*} }w'(\mathbf{m}) \overline{w'(\mathbf{m'})}|\mathbb{F}|^{\frac{d}{2}}. $$
In this case,  the third term above is negative and thus  we have
$$\sum_{\mathbf{x}\in C_{d+2}} |\widehat{w'1_{S'}}(\mathbf{x})|^2\le |\mathbb{F}|^{d+1}\sum_{\mathbf{m}\in S'} |w'(\mathbf{m})|^2 +(|\mathbb{F}|-1)|\mathbb{F}|^{\frac{d}{2}}\left|\sum_{\mathbf{m}\in S'} w'(\mathbf{m})\right|^2.$$
Proceeding as in \eqref{bor1} gives us 
$$ \sum_{\mathbf{x}\in C_{d+2}} |\widehat{w'1_{S'}}(\mathbf{x})|^2  \le \left(|\mathbb{F}|^{d+2} + |\mathbb{F}|^{\frac{d+6}{2}} |S| \right)\sum_{s\in S} w^2(s).$$
From this estimate and \eqref{Plancherel}, we obtain the statement of the second part of the proposition.

To prove the third part of the proposition, we notice from Proposition \ref{lem4.1} that  if $d\ge 3$ is odd, then
$|C_{d+2}^{\vee}(\alpha)|\le |\mathbb{F}|^{-\frac{d+3}{2}} $ for $\alpha \in \mathbb F^{d+2}\setminus \{\mathbf{0}\}$ and 
$C_{d+2}^{\vee}(\mathbf{0})=\frac{1}{|\mathbb{F}|}.$ Thus, we see that
\begin{align*}\sum_{\mathbf{x}\in C_{d+2}} |\widehat{w'1_{S'}}(\mathbf{x})|^2&= |\mathbb{F}|^{d+2} \sum_{\mathbf{m}, \mathbf{m'}\in S'}  w'(\mathbf{m}) \overline{w'(\mathbf{m'})} C_{d+2}^{\vee}(\mathbf{m}-\mathbf{m'})\\
&\le  |\mathbb{F}|^{d+1} \sum_{\mathbf{m}\in S'} |w'(\mathbf{m})|^2  + |\mathbb{F}|^{\frac{d+1}{2}} \left(\sum_{\mathbf{m}\in S'} |w'(\mathbf{m})|\right)^2\\
&\le \left(|\mathbb{F}|^{d+2} + |\mathbb{F}|^{\frac{d+5}{2}} |S| \right)\sum_{s\in S} w^2(s),\end{align*}
where the last inequality is obtained by the Cauchy-Schwarz inequality  and the definition of $w'$. 
It follows from a direct computation that  this estimate and \eqref{Plancherel} imply the conclusion of the third part of the proposition.
\end{proof}

\subsection{Proof of Theorem \ref{incidence-theorem}}
A combination of Propositions \ref{Goodsize} and \ref{Pro15} directly yields the following point-sphere incidence estimates, a special case of which is  Theorem \ref{incidence-theorem}.

\begin{theorem} 
Let $P$ be a set of points in $\mathbb F^d$ and $S$ be a set of spheres in $\mathbb F^d.$ Suppose that $w$ is a complex-valued function on $S.$ Then the following statements hold:
\begin{enumerate}
\item If $d\equiv 2 \mod{4}$ and $|\mathbb F|\equiv 3 \mod{4}$, then we have
$$\left|I_w(P, S)-|\mathbb F|^{-1} |P|\sum\limits_{s\in S} w(s) \right|\lesssim \left\{ \begin{array}{ll} |\mathbb F|^{\frac{d}{2}}\sqrt{|P|} \sqrt{\sum\limits_{s\in S} |w(s)|^2}\quad &\mbox{if}\quad |\mathbb F|^{\frac{d+2}{2}}\le |S| \le |\mathbb F|^{d+1}\\
                              |\mathbb F|^{\frac{d-2}{4}} \sqrt{|P|}\sqrt{|S|}\sqrt{\sum\limits_{s\in S} |w(s)|^2} \quad &\mbox{if}\quad |\mathbb F|^{\frac{d}{2}}\le |S| \le |\mathbb F|^{\frac{d+2}{2}}\\
                             |\mathbb F|^{\frac{d-1}{2}}\sqrt{|P|} \sqrt{\sum\limits_{s\in S} |w(s)|^2}  \quad &\mbox{if}\quad 1\le |S| \le |\mathbb F|^{\frac{d}{2}}. \end{array}\right. $$

\item If $d\equiv 0 \mod{4},$ or $d$ is even and $|\mathbb F|\equiv 1 \mod{4}$, then
$$\left|I_w(P, S)-|\mathbb F|^{-1} |P|\sum\limits_{s\in S} w(s) \right|\lesssim  \left\{ \begin{array}{ll} |\mathbb F|^{\frac{d}{2}}\sqrt{|P|} \sqrt{\sum\limits_{s\in S} |w(s)|^2} \quad &\mbox{for}\quad |\mathbb F|^{\frac{d}{2}}\le |S| \le |\mathbb F|^{d+1}\\
                              |\mathbb F|^{\frac{d}{4}} \sqrt{|P|}\sqrt{|S|}\sqrt{\sum\limits_{s\in S} |w(s)|^2} \quad &\mbox{for}\quad |\mathbb F|^{\frac{d-2}{2}}\le |S| \le |\mathbb F|^{\frac{d}{2}}\\
                            |\mathbb F|^{\frac{d-1}{2}}\sqrt{|P|} \sqrt{\sum\limits_{s\in S} |w(s)|^2}   \quad &\mbox{for}\quad 1\le |S| \le |\mathbb F|^{\frac{d-2}{2}}. \end{array}\right.$$                   

\item  If $d\ge 3$ is an odd integer, then
$$\left|I_w(P, S)-|\mathbb F|^{-1} |P|\sum\limits_{s\in S} w(s) \right|\lesssim \left\{ \begin{array}{ll} |\mathbb F|^{\frac{d}{2}}\sqrt{|P|} \sqrt{\sum\limits_{s\in S} |w(s)|^2} \quad &\mbox{for}\quad |\mathbb F|^{\frac{d+1}{2}}\le |S| \le |\mathbb F|^{d+1}\\
                              |\mathbb F|^{\frac{d-1}{4}} \sqrt{|P|}\sqrt{|S|}\sqrt{\sum\limits_{s\in S} |w(s)|^2} \quad &\mbox{for}\quad |\mathbb F|^{\frac{d-1}{2}}\le |S| \le |\mathbb F|^{\frac{d+1}{2}}\\
                              |\mathbb F|^{\frac{d-1}{2}}\sqrt{|P|} \sqrt{\sum\limits_{s\in S} |w(s)|^2}  \quad &\mbox{for}\quad 1\le |S| \le |\mathbb F|^{\frac{d-1}{2}}. \end{array}\right. $$            
\end{enumerate}
\end{theorem}

\begin{remark}  If the sphere set is of very large size, say bigger than $|\mathbb{F}|^{{(d+2)}/{2}}, ~ |\mathbb{F}|^{{d}/{2}}, ~ |\mathbb{F}|^{{d+1}/{2}}$, we recover the result given  by Cilleruelo, Iosevich, Lund, Roche-Newton, and Rudnev \cite{CILRR17}, and independently by Pham, Phuong and Vinh \cite{PPV}. When $S$ has medium size, we obtain the error terms $|\mathbb{F}|^{{(d-2)}/{4}}|P|^{1/2}|S|,$ $~|\mathbb{F}|^{{d}/{4}}|P|^{1/2}|S|, ~|\mathbb{F}|^{{(d-1)}/{4}}|P|^{1/2}|S|$ in corresponding cases which are much better than the Cauchy-Schwarz bound $|\mathbb{F}|^{{(d-2)}/{2}}|P|^{1/2}|S|$ attained by using the fact that any two distinct spheres intersect in at most $\lesssim |\mathbb{F}|^{d-2}$ elements. It is necessary to mention that for sphere sets of medium size, our results not only give better upper bounds but also tell us about lower bounds. 
\end{remark}
\section{Constructions and Remarks}\label{moreds}
\subsection{Sharpness of Theorem \ref{incidence-theorem}}\label{chat}
We start with the following simple lemma. 
\begin{lemma} \label{example1} Let $d=4\ell+2$ for $\ell\in \mathbb N$ and $|\mathbb F|\equiv 3 \mod{4}.$ 
There exist a set $P$ of points in $\mathbb F^d$ and a set $S$ of spheres in $\mathbb F^d$ such that $|S|\sim |\mathbb F|^{d/2}, ~|P||S|\sim |\mathbb F|^{d+1},$ and $I(P, S)=0.$
\end{lemma}
\begin{proof}
We begin by reviewing the definition of mutually orthogonal null vectors. 
Let $V$ be a set of vectors in $\mathbb F^n$. Suppose that $V=\{v_1, \ldots, v_k\}$. We say that vectors in $V$ are mutually orthogonal null vectors if $v_i\cdot v_j=0$ and $v_i\ne \mathbf{0}$ for any $1\le i, j\le k$.  If $n=4\ell$ with $\ell\in \mathbb{N}$, it has been proved by Hart, Iosevich, Koh, and Rudnev (\cite{hart}, Lemma 5.1) that there always exist $n/2$ mutually orthogonal null vectors. 

\begin{remark} We claim that when $d=4\ell+2$ and $|\mathbb F|\equiv 3\mod 4$, it is impossible to find $d/2$ mutually orthogonal null vectors in $\mathbb F^d$. 
Otherwise, we would have $d/2$ mutually orthogonal null vectors $v_1,v_2, \ldots, v_{d/2}$ in $\mathbb F^d.$ Note that  $v_i\cdot v_j=0$ for all $i, j=1,2, \ldots, d/2$. In particular, $||v_i||=v_i\cdot v_i=0$ for all $i.$
For each $i=1,2, \ldots, d/2,$ define $w_i:=(v_i, 0, 0).$ Then $w_1, w_2,\ldots, w_{d/2}$ are $d/2$ mutually orthogonal null vectors in $\mathbb F^{d+2}.$
Now define
$$ H=\mathtt{Span}(w_1, \ldots, w_{d/2})+\{(0, \ldots, 0, 0, s)\in \mathbb F^{d+2}\colon s\in \mathbb F\}.$$
Here, $\mathtt{Span}(w_1, \ldots, w_{d/2})$ denotes a set of vectors spanned by $w_1, \ldots, w_{d/2}.$ In other words,  
$$\mathtt{Span}(w_1, \ldots, w_{d/2}):=\left\{\sum_{i=1}^{d/2} a_i w_i: a_i\in \mathbb F, i=1,2, \ldots, d/2\right\}.$$
It is not hard to check that $H$ is a $(d+2)/2$ dimensional subspace lying on the cone $C_{d+2}.$ 
Letting $n=d+2$, we see that $|H|=|\mathbb F|^{n/2}$ and $H\subset C_n$.  Applying the inequality \eqref{neko1} with $p=2$ and $k=n/2$, a necessary condition for $R^*_{C_n}(2\to r)\lesssim 1$ to hold is given by the condition 
$ r\ge \frac{2n}{n-2}.$ However, this contradicts  the result of Theorem \ref{mainthm}, that is $R_{C_n}^*\left(2\to \frac{2n+4}{n}\right)\lesssim 1.$
\end{remark}

We now prove Lemma \ref{example1}. Since $d\equiv 2 \mod{4},$ we have $d-2=4\ell$ for some $\ell\in \mathbb N.$ Hence,   there are $(d-2)/2$ mutually orthogonal null vectors in $\mathbb F^{d-2}\times \{(0, 0)\},$ say $ w_1, \ldots, w_{(d-2)/2}.$ 
Let $U\subset \mathbb F^2$ be the set of all points on $(|\mathbb F|+1)/2$ concentric circles, centered at the origin, of radii in 
$\{r_i\in \mathbb F^*: i=1,2, \ldots, (|\mathbb F|+1)/2\}$. Let  $R:=\mathbb F\setminus \{r_1, \ldots, r_{(|\mathbb F|+1)/2}\}$. 
Define
\[P=\mathtt{Span}(w_1, \ldots, w_{(d-2)/2})+\{(0, \ldots, 0, x_1, x_2)\colon (x_1,x_2)\in U\},\]
where 
$$\mathtt{Span}(w_1, \ldots, w_{(d-2)/2}):=\left\{\sum_{i=1}^{(d-2)/2} a_i w_i: a_i\in \mathbb F, i=1,2, \ldots, (d-2)/2\right\}.$$
Set $B=\mathtt{Span}(w_1, \ldots, w_{(d-2)/2}),$ and  let $S$ be the set of spheres of radii in $R$, centered at points in $B.$ One can check that 
$|S|=|B||R|= |\mathbb F|^{(d-2)/2}(|\mathbb F|-1)/2 \sim |\mathbb F|^{d/2}$ and $|P|=|\mathbb F|^{{(d-2)}/{2}}\, \frac{(|\mathbb F|+1)}{2} (|\mathbb F|+1)\sim |\mathbb F|^{(d+2)/2 }.$  Hence, we have  $|P||S|\sim |\mathbb F|^{d+1}.$ It remains to show that $I(P, S)=0.$ To prove this, taking an arbitrary $p\in P$ and fixing $b\in B$, it suffices to show that $||p-b||\notin R.$  Now, fix $x\in U,$ and consider all $p=w+x$, as $w$ ranges over $B.$ Then, we have
$$ ||p-b||=||w-b||+ ||x||=0+ ||x||=||x||\notin R,$$
as required.
Here, the fact that $||w-b||=0$ follows by the following observation.
Any element of $B$ is a linear combination of mutually orthogonal null vectors $w_1, \ldots, w_{(d-2)/2}$ and  $B$ is a subspace, and hence for any $w, b\in B$, we have  $w-b\in B$ and  $0=(w-b)\cdot (w-b)=||w-b||.$
\end{proof}
It follows from the proof of Lemma \ref{example1} that we can choose a family  $S$ of spheres of much smaller size so that $|P||S|\sim |\mathbb{F}|^{d+1}$ and $I(P, S)=0$. Indeed, for instance, instead of choosing $(d-2)/2$ mutually orthogonal null vectors, we can choose $(d-6)/2$ mutually orthogonal null vectors in $\mathbb F^{d-6}\times \{(0, 0, 0, 0, 0, 0)\},$ say $ w_1, \ldots, w_{(d-6)/2}.$ This can be done because $d-6=4\ell$ for some $\ell\in \mathbb N.$

Define $B$ as a set of vectors spanned by $w_1, \ldots, w_{(d-6)/2}.$ In other words,  
$$ B:=\mathtt{Span}(w_1, \ldots, w_{(d-6)/2})=\left\{\sum_{i=1}^{(d-6)/2} a_i w_i: a_i\in \mathbb F, i=1,2, \ldots, (d-6)/2\right\}.$$
Let $U\subset \mathbb F^6$ be the set of all points on $(|\mathbb F|+1)/2$ concentric spheres, centered at the origin, of radii in 
$\{r_i\in \mathbb F^*: i=1,2, \ldots, (|\mathbb F|+1)/2\}$. Let  $R:=\mathbb F\setminus \{r_1, \ldots, r_{(|\mathbb F|+1)/2}\}$. Similarly, define
\[P=\mathtt{Span}(w_1, \ldots, w_{(d-6)/2})+\{(0, \ldots, 0, x_1, x_2, x_3, x_4, x_5, x_6)\colon (x_1,x_2, x_3, x_4, x_5, x_6)\in U\}.\]
Let $S$ be the set of spheres of radii in $R$, centered at points in $B.$ One can check that 
$|S|=|B||R|= |\mathbb F|^{(d-6)/2}(|\mathbb F|-1)/2 \sim |\mathbb{F}|^{{(d-4)}/{2}}$ and $|P|\sim|\mathbb F|^{{(d-6)}/{2}} (|\mathbb F|+1) |\mathbb F|^5,$ where we have used the fact that any sphere of non-zero radius in $\mathbb{F}^6$ contains around $|\mathbb{F}|^5$ points. It is not hard to see that $I(P, S)=0$ and $|P||S|\sim |\mathbb F|^{(d+1)}$, as required. Repeating this process,  the following theorem is attained. 
\begin{theorem}[Sharpness of Theorem \ref{incidence-theorem} (1)]
Let $d=4\ell+2$ for $\ell\in \mathbb N$ and $|\mathbb F|\equiv 3 \mod{4}.$ For any $k\in \mathbb{N}$ with $d\ge 4k$, there exist sets $S$ and $P$ in $\mathbb{F}^d$ with $|S|\sim |\mathbb{F}|^{\frac{d-4k}{2}}$ and $|P||S|\sim |\mathbb{F}|^{d+1}$ such that $I(P, S)=0$.
\end{theorem} 

Using the same idea, we have the following sharpness of Theorem \ref{incidence-theorem} (2) and (3) whose proofs we will omit.

\begin{theorem}[Sharpness of Theorem \ref{incidence-theorem} (2)] 
Assume that  either $d=4\ell$ for $\ell\in \mathbb N$  or $d=4\ell+2$ for $\ell\in \mathbb N$ and $|\mathbb{F}|\equiv 1\mod 4$. For any $k\in \mathbb{N}$ with $k\ge 1$ and $d+2\ge 4k$, there exist a set $P$ of points in $\mathbb F^d$ and a set $S$ of spheres in $\mathbb F^d$ such that $|S|\sim |\mathbb F|^{\frac{d-(4k-2)}{2}}, ~|P||S|\sim |\mathbb F|^{d+1},$ and $I(P, S)=0.$
\end{theorem}
\begin{theorem}[Sharpness of Theorem \ref{incidence-theorem} (3)]The following statements hold.

\begin{enumerate}
\item Suppose $d\ge 3$ is an odd integer and $|\mathbb{F}|\equiv 1\mod 4$. For any $k\in \mathbb{N}$ with  $d+1\ge 2k$, there exist a set $P$ of points in $\mathbb F^d$ and a set $S$ of spheres in $\mathbb F^d$ such that $|S|\sim |\mathbb F|^{\frac{d-(2k+1)}{2}}, ~|P||S|\sim |\mathbb F|^{d+1},$ and $I(P, S)=0.$
\item Suppose $d=4\ell+1$ for $\ell\in \mathbb N$ and $|\mathbb{F}|\equiv 3\mod 4$. For any $k\in \mathbb{N}$ with $k\ge 1$ and $d+1\ge 4k$, there exist a set $P$ of points in $\mathbb F^d$ and a set $S$ of spheres in $\mathbb F^d$ such that $|S|\sim |\mathbb F|^{\frac{d-(4k-1)}{2}}, ~|P||S|\sim |\mathbb F|^{d+1},$ and $I(P, S)=0.$
\item Suppose $d=4\ell+3$ for $\ell\in \mathbb N$ and $|\mathbb{F}|\equiv 3\mod 4$. For any $k\in \mathbb{N}$ with $d-1\ge 4k$, there exist a set $P$ of points in $\mathbb F^d$ and a set $S$ of spheres in $\mathbb F^d$ such that $|S|\sim |\mathbb F|^{\frac{d-(4k+1)}{2}}, ~|P||S|\sim |\mathbb F|^{d+1},$ and $I(P, S)=0.$
\end{enumerate}
\end{theorem}
\subsection{Remark on the Erd\H{o}s-Falconer conjecture}\label{RemarkE}
Let us first recall the statement of the problem. The Erd\H{o}s-Falconer distance problem over finite fields asks for the smallest number $\alpha$ such that  if $E\subset \mathbb{F}^d$ and $|E|\ge C |\mathbb{F}|^{\alpha}$  then the distance set $\Delta(E):=\{||x-y||\colon x, y\in E\}$ contains the whole field $\mathbb{F}$ or covers a positive proportion of all distances. It has been proved that in odd dimensions, except $d=4\ell-1$ and $|\mathbb{F}|\equiv 3\mod 4$, the exponent ${(d+1)}/{2}$ is sharp even one wishes to cover a positive proportion of all possible distances. In even dimensions with $|\mathbb{F}|\equiv 3\mod 4$, it is conjectured that the right exponent should be $d/2$.  We refer the interested reader to \cite{hart, KPV18} for more details.

Theorem \ref{incidence-theorem} (1) with $w(s)=1$ for all $s\in S$ says that
with the condition that $|S|\le |\mathbb F|^{d/2}$ we have
\begin{equation}\label{FakeAss}\left\vert I(P, S)-\frac{|P||S|}{|\mathbb F|}\right\vert\le C_1  |\mathbb F|^{\frac{d-1}{2}}\sqrt{|P||S|},\end{equation}
where $C_1$ is a fixed universal constant. Choose a constant $C>1$ such that
\begin{equation}\label{conC} (C-1)\sqrt{C} > C_1.\end{equation} 
We claim that if the condition on $|S|$ can be relaxed to $|S|\le q^{(d+2)/2}$, then one can settle the Erd\H{o}s-Falconer distance conjecture in the case when $d\equiv 2\mod{4}$ and $|\mathbb F|\equiv 3 \mod{4}$.\\

To prove our claim, let  $P$ be an arbitrary set in $\mathbb F^d$ such that $|P|= C|\mathbb F|^{d/2}.$  For each $x\in P$, let $S_x$ denote the set of spheres centered at $x$ of radii in $\Delta_x(P):=\{||x-y||\colon y\in P\}.$ Let $S=\cup_{x}S_x$. It is clear that $|S|=\sum_{x\in P}|\Delta_x(P)|$. If $|S|\ge |\mathbb F||P|/C$, then by the pigeon-hole principle, there exists a point $x_0\in P$ such that $|\Delta_{x_0}(P)|\ge |\mathbb F|/C$, and we are done. Hence, we can assume that 
$|S|\le |\mathbb F||P|/C= 
|\mathbb F|^{(d+2)/2}$.  Using the inequality \eqref{FakeAss},    we have
\[|P|^2=I(P, S)\le \frac{|\mathbb F||P|^2}{C|\mathbb F|}+\frac{C_1|\mathbb F|^{(d-1)/2}|P|^{1/2}|\mathbb F|^{1/2}|P|^{1/2}}{\sqrt{C}}.\]
This implies that 
\[(1-1/C)|P|\le \frac{C_1}{\sqrt{C}}|\mathbb F|^{d/2},\]
which contradicts \eqref{conC} since $|P|=C|\mathbb F|^{d/2}.$

\section*{Acknowledgements}
 D. Koh was supported by Basic Science Research Program through the National
Research Foundation of Korea(NRF) funded by the Ministry of Education, Science
and Technology(NRF-2018R1D1A1B07044469). T. Pham was supported by Swiss National Science Foundation grant P2ELP2175050.

 \bibliographystyle{amsplain}

\end{document}